\documentclass[11pt]{amsart}
\usepackage[a4paper, footskip=1cm, headheight = 16pt, top=3.7cm, bottom=3cm,  right=2.5cm,  left=2.5cm ]{geometry}
\usepackage[colorlinks,linkcolor=blue,citecolor=red]{hyperref}
\usepackage[T1]{fontenc}
\usepackage[utf8]{inputenc}
\usepackage[noadjust]{cite}
\usepackage{tikz}
\usetikzlibrary{decorations.pathmorphing}
\tikzset{snake it/.style={decorate, decoration=snake}}
\usepackage{listings}
\usepackage{cite}
\usepackage{mathrsfs}
\usepackage{amstext,amsfonts,amsthm,amsmath,amssymb}
\usepackage{graphicx,tikz}
\usepackage{comment}
\usepackage{url}
\usepackage{xspace}
\usepackage{todonotes}
\usepackage[shortlabels]{enumitem}
\usepackage[absolute]{textpos}
\usepackage{thm-restate}
\usepackage[foot]{amsaddr}
\usepackage{xcolor}
\colorlet{RED}{red}
\makeatletter
\DeclareRobustCommand*\cal{\@fontswitch\relax\mathcal}
\makeatother

\def\cqedsymbol{\ifmmode$\lrcorner$\else{\unskip\nobreak\hfil
\penalty50\hskip1em\null\nobreak\hfil$\lrcorner$
\parfillskip=0pt\finalhyphendemerits=0\endgraf}\fi}

\newcommand{\C}{\mathcal{C}}

\newcommand{\sep}{\text{sep}}
\newcommand{\tw}{\text{tw}}
\newcommand{\rw}{\text{rw}}

\newtheorem{lemma}{Lemma}[section]

\newtheorem{corollary}[lemma]{Corollary}
\newtheorem{theorem}[lemma]{Theorem}

\newtheorem{conjecture}[lemma]{Conjecture}
\theoremstyle{definition}

\def\dd{\hbox{-}}

\newcommand{\mylongtitle}[1]{%
  \ifodd\value{page}%
    \protect\parbox{0.97\linewidth}{#1}\hfill%
  \else%
    \hfill\protect\parbox{0.97\linewidth}{#1}%
  \fi%
}

\title[Induced subgraphs and tree decompositions I.]{ {Induced subgraphs and tree decompositions
I.\\ Even-hole-free graphs of bounded degree}}

\author{Tara Abrishami$^{\ast \mathsection}$}
\author{Maria Chudnovsky$^{\ast \dagger}$}
\author{Kristina Vu\v{s}kovi\'c $^{\ast\ast \ddagger}$}
\address{$^{\ast}$Princeton University, Princeton, NJ, USA}
\address{$^{\ast\ast \ddagger}$School of Computing, University of Leeds, UK}
\address{$^{\mathsection}$Supported by NSF Grant DMS-1763817 and
     NSF-EPSRC Grant DMS-2120644.}
\address{$^{\dagger}$ Supported by NSF Grant DMS-1763817. This material is based upon work supported by, or in part by, the U.S. Army Research Laboratory and the U. S. Army Research Office under grant number W911NF-16-1-0404.}
\address{$^{\ddagger}$ Partially supported by EPSRC grant EP/N0196660/1.}

\date{\today}

\begin{document}
\maketitle

\begin{abstract}
Treewidth is a parameter that emerged from  the study of minor closed classes of graphs (i.e. classes closed under vertex and edge deletion, and edge contraction). 
It in some sense describes the global structure of a graph.
Roughly, a graph has treewidth $k$ if it can be decomposed by a sequence of noncrossing cutsets of size at most $k$ into pieces of size at most $k+1$.
The study of hereditary graph classes (i.e. those closed under vertex deletion only) reveals a different picture, where cutsets
that are not necessarily bounded in size (such as star cutsets, 2-joins and their generalization) are required to decompose the graph into simpler pieces that
are structured but not necessarily bounded in size.
A number of such decomposition theorems are known for complex hereditary graph classes, including  even-hole-free graphs, perfect graphs and others.
These theorems  do not describe the global structure in the sense that a tree decomposition does, since the cutsets guaranteed by them  are far from being noncrossing. They are also of limited use in algorithmic applications.


We show that in the case of even-hole-free graphs of bounded degree the  cutsets described in the previous paragraph can be partitioned into a bounded number of well-behaved collections. This allows us to prove that 
even-hole-free graphs with bounded degree have bounded treewidth, resolving a conjecture of Aboulker, Adler, Kim, Sintiari and Trotignon
[arXiv:2008.05504].  As a consequence, it follows that 
many algorithmic problems can be solved in polynomial time for this class, and that even-hole-freeness is testable in the bounded degree graph model of property testing.
In fact we prove our results for a larger class of graphs, namely the class of $C_4$-free odd-signable graphs with bounded degree.
\end{abstract}

\section{Introduction}

All graphs in this paper are finite and simple. 
A \emph{hole} of a graph $G$ is an induced cycle of $G$ of length at least four. 
 A graph is {\em even-hole-free} if it has no hole with an even number of vertices. 
 
Even-hole-free graphs have been studied extensively; see \cite{Vuskovic2010} for a survey. 
The first polynomial time recognition algorithm for this class of graphs was obtained in \cite{Conforti2002Even-hole-freeRecognition}. This  algorithm is 
 based on a  decomposition theorem  from
\cite{Conforti2002Even-hole-freeTheorem} that uses 2-joins and star, double star, and triple star cutsets to decompose the graph into simpler pieces.
Later, a stronger decomposition theorem, using only star cutsets and 2-joins,
was obtained in  \cite{daSilva2013Decomposition2-joins}, leading to 
a faster recognition algorithm. Further improvements resulted in
the best currently known algorithm with running  time
${\cal O}(n^9)$ \cite{Chang2015EHFrecognition,Thorup2020}.
This progress required deep insights into the behavior  of even-hole-free
graphs; however the  global structure of graphs in this class is still not well
understood. Moreover, there are several natural optimization problems
whose
complexity for this class remains unknown (among those 
are the vertex coloring problem and the maximum weight stable set problem).
The key difficulty is to make use of star cutsets, and in particular to
understand how several star cutsets in a given graph interact. In this paper
we address this problem, by showing that star cutsets in an even-hole-free graph
of bounded degree can be partitioned into a bounded number of well-behaved
collections, which in turn allows us to bound the treewidth of such graphs.

Let $G = (V, E)$ be a graph. A \emph{tree decomposition} $(T, \chi)$ of $G$ is a tree $T$ and a map $\chi: V(T) \to 2^{V(G)}$ such that the following hold: 
\begin{enumerate}[(i)]
\itemsep -.2em
    \item For every $v \in V(G)$, there exists $t \in V(T)$ such that $v \in \chi(t)$. 
    
    \item For every $v_1v_2 \in E(G)$, there exists $t \in V(T)$ such that $v_1, v_2 \in \chi(t)$.
    
    \item For every $v \in V(G)$, the subgraph of $T$ induced by $\{t \in V(T) \mid v \in \chi(t)\}$ is connected.
\end{enumerate}

If $(T, \chi)$ is a tree decomposition of $G$ and $V(T) = \{t_1, \hdots, t_n\}$, the sets $\chi(t_1), \hdots, \chi(t_n)$ are called the \emph{bags of $(T, \chi)$}.  The \emph{width} of a tree decomposition $(T, \chi)$ is $\max_{t \in V(T)} |\chi(t)|-1$. The \emph{treewidth} of $G$, denoted $\tw(G)$, is the minimum width of a tree decomposition of $G$.

Many NP-hard algorithmic problems can be solved in polynomial time in graphs with bounded treewidth. For a full discussion, see \cite{Bodlaender1988DynamicTreewidth}. While tree decomposition{ s}, and classes of graphs of bounded treewidth, play an important role in the study of graphs with forbidden minors, the problem of connecting tree decompositions with forbidden induced subgraphs has so far remained open. Clearly, in order to get a class of bounded treewidth, one needs to forbid, for example, large cliques, large complete bipartite graphs, large walls, and the line graphs of large walls.  However,  all of these obstructions{ , except for large cliques,} contain even holes. Further, in \cite{Silva2010}, a bound on the treewidth of planar even-hole-free graphs was proven. On the other hand,  \cite{SD} contains a construction of a family of even-hole-free graphs with no $K_4$, and with unbounded treewidth.  The graphs in this construction have both unbounded degree and contain large clique minors. In \cite{Aboulker2020OnGraphs} it was examined  whether both of these are necessary. They show that any graph that excludes a fixed graph as a minor either has small treewidth or contains (as an induced subgraph) a large wall or the line graph of a large wall. This implies that even-hole-free graphs that exclude a fixed graph as a minor have bounded treewidth (generalizing the result of \cite{Silva2010}).  Furthermore,  the following conjecture was made (and proved for subcubic graphs) in \cite{Aboulker2020OnGraphs}:

\begin{conjecture} \label{mainconj}
  For every $\delta \geq 0$ there exists $k$ such that  even-hole-free graphs
  with maximum degree at most $\delta$ have treewidth at most $k$. 
\end{conjecture}

The main result of the present paper is the proof of Conjecture \ref{mainconj}, in fact, the following slight strengthening of it.
We {\em sign} a graph $G$ by assigning $0,1$ weights to its edges.
A graph $G$ is {\em odd-signable} if there exists a signing such that
every triangle and every hole in $G$ has odd weight. Thus even-hole-free
graphs are a subclass of odd-signable graphs.

\begin{theorem} \label{mainthm}
  For every $\delta \geq 0$ there exists $k$ such that  $C_4$-free
  odd-signable graphs   with maximum degree at most $\delta$ have treewidth at
  most $k$. 
  \end{theorem}

It follows from Theorem \ref{mainthm} that vertex coloring, maximum stable set, and 
many other NP-hard algorithmic problems can be solved in polynomial time for
even-hole-free graphs with bounded maximum degree.
Another consequence of Theorem \ref{mainthm} is that even-hole-freeness is testable in the bounded degree graph model of property testing,
since even-hole-freeness is expressible in monadic second-order logic with modulo counting (CMSO) and CMSO is testable on bounded treewidth \cite{Adler2018}.
See \cite{Aboulker2020OnGraphs} for an excellent survey that motivates the study of Conjecture \ref{mainconj} and surrounding problems, and in particular contains a detailed discussion of property testing algorithms.

\subsection{Outline of the proof of Theorem  \ref{mainthm}}

A graph $G$ has bounded treewidth if and only if every connected component of $G$ has bounded treewidth. Therefore, we prove that connected $C_4$-free odd-signable graphs with bounded degree have bounded treewidth.

In \cite{Harvey2017ParametersTreewidth}, a number of parameters tied to treewidth are discussed. Let $G$ be a graph, let $c \in [\frac{1}{2}, 1)$, and let $k$ be a nonnegative integer. For $S \subseteq V(G)$, a \emph{$(k, S, c)^*$-separator} is a set $X \subseteq V(G)$ with $|X| \leq k$ such that every component of $G \setminus X$ contains at most $c|S|$ vertices of $S$. The \emph{separation number} $\sep_c^*(G)$ is the minimum $k$ such that there exists a $(k, S, c)^*$-separator for every $S \subseteq V(G)$. The separation number is related to treewidth through the following lemma. 

\begin{lemma}[\cite{Harvey2017ParametersTreewidth}]
\label{lemma:sep_tw}
For every graph $G$ and for all $c \in [\frac{1}{2}, 1)$, the following holds:
\begin{equation*} \sep_c^*(G) \leq \tw(G) + 1 \leq \frac{1}{1-c} \sep_c^*(G). \end{equation*}
\end{lemma}

 A set $S \subseteq V(G)$ is \emph{$d$-bounded} if there exist $v_1, \hdots, v_{d'}$, with $d' \leq d$, such that $S \subseteq N^d[v_1] \cup \hdots \cup N^d[v_{d'}]$. 
 { 
 For a graph $G$ and weight function $w$ on its vertices,
  if $X$ is a subgraph of $G$ or a subset of $V(G)$, then $w(X)$ is the sum of the weights of vertices in $X$. }
 Let $G$ be a graph and let $w: V(G) \to [0, 1]$ be a weight function of $G$ such that $w(G) = 1$. By $w^{\max}$ we denote the maximum weight of a vertex in $G$. 
A set $Y \subseteq V(G)$ is a \emph{$(w, c, d)$-balanced separator} of $G$ if $Y$ is $d$-bounded and $w(Z) \leq c$ for every component $Z$ of $G \setminus Y$. The following lemma shows that if $G$ is a graph with maximum degree $\delta$ and $G$ has a $(w, c, d)$-balanced separator for every weight function $w:V(G) \to [0, 1]$ with $w(G) = 1$, then $G$ has bounded treewidth. 
 
 \begin{lemma} 

\label{lemma:boundedtw}
Let $\delta, d$ be positive integers with $\delta \leq d$, let $c \in [\frac{1}{2}, 1)$, and let $\Delta(d) = d + d\delta + d\delta^2 + \hdots + d \delta^d$. Let $G$ be a graph with maximum degree $\delta$. Suppose that for every $w: V(G) \to [0, 1]$ with $w(G) = 1$ and $w^{\max} < \frac{1}{\Delta(d)}$, $G$ has a $(w, c, d)$-balanced separator.  Then, $\tw(G) \leq \frac{1}{1-c} \Delta(d)$. \end{lemma}
\begin{proof}
Note that $\Delta(d)$ is an upper bound for the size of a $d$-bounded set in $G$.  Let $S \subseteq V(G)$. If $|S| \leq \Delta(d)$, then $S$ is a $(\Delta(d), S, c)^*$-separator of $G$. Now, assume $|S| > \Delta(d)$. Let $w_S:V(G) \to [0, 1]$ be such that $w_S(v) = \frac{1}{|S|}$ for $v \in S$ and $w_S(v) = 0$ for $v \in V(G) \setminus S$. Then, $w_S(G) = 1$ and $w_S^{\max} < \frac{1}{\Delta(d)}$, so $G$ has a $(w_S, c, d)$-balanced separator. Specifically, for all $S \subseteq V(G)$ such that $|S| > \Delta(d)$, there exists a set $X$ such that $|X| \leq \Delta(d)$, and $w_S(Z) \leq c$ for all components $Z$ of $G \setminus X$. It follows that $X$ is a $(\Delta(d), S, c)^*$-separator of $G$. Therefore, $G$ has  a $(\Delta(d), S, c)^*$-separator for every $S \subseteq V(G)$. It follows that $\sep_c^*(G) \leq \Delta(d)$, and by Lemma \ref{lemma:sep_tw}, $\tw(G) \leq \frac{1}{1-c} \Delta(d)$. 
\end{proof}

 In this paper, we prove that connected $C_4$-free odd-signable graphs with bounded degree have bounded treewidth. Specifically, we prove the following theorem:
\begin{restatable}{theorem}{mainthm}
\label{thm:mainthm}
Let $\delta, d$ be positive integers. Let $G$ be a connected $C_4$-free odd-signable graph with maximum degree $\delta$ and let $w:V(G) \to [0, 1]$ be a weight function such that $w(G) = 1$. Let $f(2, \delta) = 2(\delta + 1)^2 + 1$, and let $c \in [\frac{1}{2}, 1)$. {  Assume that} $d \geq 49\delta + {  4}f(2, \delta)\delta -4$ and $(1-c) + [w^{\max} + 3f(2, \delta)\delta 2^{\delta}(1-c) + 2(\delta - 1)2^{\delta}(1-c)](\delta + \delta^2) < \frac{1}{2}$.  Then, $G$ has a $(w, c, d)$-balanced separator.
\end{restatable}

We can then prove our main result: 
\begin{theorem}\label{thm:result_boundedtw}
Let $\delta$ be a positive integer and let $G$ be a connected $C_4$-free odd-signable graph with maximum degree $\delta$. Then, there exists $c \in [\frac{1}{2}, 1)$ and positive integer $d \geq \delta$ such that $\tw(G) \leq \frac{1}{1-c}(d + d\delta + d\delta^2 + \hdots + d\delta^d)$. 
\end{theorem}
\begin{proof}
{  Let $f(2, \delta) = 2(\delta + 1)^2 + 1$. Let $d$ be an integer such that $d \geq 49\delta + 4f(2, \delta)\delta - 4$, and let
 $\Delta(d) = d + d\delta + d\delta^2 + \hdots + d\delta^d$. 
 Note that there exists 
 $c \in [\frac{1}{2}, 1)$ such that $(1-c) + [\frac{1}{\Delta(d)} + 3f(2, \delta)\delta2^{\delta}(1-c) + 2(\delta-1)2^{\delta}(1-c)](\delta + \delta^2) < \frac{1}{2}$. 
 Let 
 $w:V(G) \to [0, 1]$ be a weight function of $G$ such that $w(G) = 1$ and $w^{\max} < \frac{1}{\Delta(d)}$. 
 Then by Theorem \ref{thm:mainthm}, $G$ has a $(w,c,d)$-balanced separator. The result now
 follows  from  Lemma \ref{lemma:boundedtw}. }
\end{proof}

Let us now discuss the main ideas of the proof of {  Theorem} \ref{thm:mainthm}. We will
give precise definitions of the concepts used below later in the paper; the goal of the next few paragraphs is just to give the reader a road map of where we are going.
 {A \emph{separation} (or {\em decomposition})  of a graph $G$ is a triple of disjoint vertex sets $(A, C, B)$ such that $A \cup C \cup B = V(G)$ and there are no edges from $A$  to $B$. To
``decompose along $(A,C,B)$'' means to delete $A$.}
Usually, to prove a result that a certain graph family has bounded treewidth, one attempts to construct a collection of ``non-crossing separations'', which  roughly  means that the separations
``cooperate'' with each other, and the pieces that are obtained when the graph
is simultaneously decomposed by all the separtions in the collection
``line up'' to form a tree structure. Such collections of separations are
called ``laminar.''

In the case of $C_4$-free odd-signable graphs, there is a natural family of separations  to turn to,  {given by Lemmas \ref{lemma:proper_wheel_forcer}, \ref{lemma:univ_wheel_forcer}, and  \ref{lemma:short_pyramid_forcer}.
A key point here is that all the decompositions above
are forced by the presence of certain induced subgraphs that we call
``forcers.'' In essence it is shown that the corresponding decomposition of the
forcer extends to the whole graph, and when the graph is decomposed along the
decomposition, part of the forcer is removed.}

Unfortunately, the decompositions above are very far from being non-crossing, and therefore
we cannot use them in traditional ways to get tree decompositions. What turns out to be true, however, is that, due to the bound on the maximum degree of the graph, this collection of decompositions can be partitioned into a bounded number
of laminar collections $X_1, \ldots, X_p$ (where $p$ depends on the maximum degree).
We can then proceed as follows. Let $G$ be a connected $C_4$-free odd-signable graph
with maximum degree $\delta$ and let $w:V(G) \to [0, 1]$ be such that
$w(G) = 1$.  {In view of Lemma \ref{lemma:sep_tw}},  to prove Theorem \ref{thm:mainthm}, we would like to show that for
  certain $c$ and $d$,  $G$ has a $(w, c, d)$-balanced separator;
we may  assume that no such separator exists.
  We first decompose
  $G$, simultaneously, by all the decompositions in $X_1$. Since $X_1$ is a
  laminar collection, by Lemma~\ref{lemma:noncrossing-separations-td} this gives a tree decomposition of $G$, and we
  identify one of the bags of this decomposition as the ``central bag'' for $X_1$; denote it by $\beta_1$. Then,  $\beta_1$ {  corresponds to} 
  an induced subgraph   of $G$, and
  we can show that $\beta_1$ has no  $(w_1,c,d_1)$-balanced separator for certain $w_1$ and $d_1$ that depend on $w$ and $d$.
We next  focus on $\beta_1$, and decompose it using $X_2$, and so on.
At step $i$, having decomposed by $X_1,\ldots, X_i$, we focus on a
central bag $\beta_i$  that does not have a  $(w_i,c,d_i)$-separator
for suitably chosen $w_i, d_i$. 

 {The fact that all the separations at play come from forcers
 ensures that}  at step $i$,
after decomposing by $X_1, \ldots, X_i$, none of the forcers that were
``responsible'' for the decompositions in $X_1, \ldots, X_i$ is present in
the central bag $\beta_i$ (as part of each such forcer was removed in the
decomposition process). It then follows that when we reach $\beta_p$, all
we are left with is a ``much simpler'' graph (one that contains no forcers),
where we can find
a $(w_p,c,d_p)$-balanced separator directly, thus obtaining a contradiction.

The remainder of the paper is devoted to proving Theorem \ref{thm:mainthm}. In Section \ref{sec:prelims}, we review key definitions and preliminaries. In Section \ref{sec:laminar}, we define laminar collections of separations, and describe a tree decomposition corresponding to a laminar collection of separations. In Section \ref{sec:clique}, we prove results about clique cutsets and balanced separators. In Sections \ref{sec:forcers} and \ref{sec:twin_wheel_clean}, we define forcers and prove results about forcers, star cutsets, and balanced separators. In Section \ref{sec:no_sc}, we prove a bound on separation number in graphs with no star cutset. Finally, in Section \ref{sec:final}, we prove Theorem \ref{thm:mainthm}.

\subsection{Terminology and notation}
\label{sec:prelims}

Let $G$ and $H$ be graphs. We say that $G$ \emph{contains} $H$ if $G$ has an induced subgraph isomorphic to $H$. We say that $G$ is \emph{$H$-free} if $G$ does not contain $H$. If $\mathcal{H}$ is a set of graphs, we say that $G$ is {\em $\mathcal{H}$-free} if $G$ is $H$-free for every $H \in \mathcal{H}$. For $X \subseteq V(G)$, $G[X]$ denotes the subgraph of $G$ induced by $X$, and $G \setminus X = G[V(G) \setminus X]$. In this paper, we use induced subgraphs and their vertex sets interchangeably. Let $v \in V(G)$. The \emph{open neighborhood of $v$}, denoted $N(v)$, is the set of all vertices in $V(G)$ adjacent to $v$. The \emph{closed neighborhood of $v$}, denoted $N[v]$, is $N(v) \cup \{v\}$. Let $X \subseteq V(G)$. The \emph{open neighborhood of $X$}, denoted $N(X)$, is the set of all vertices in $V(G) \setminus X$ with a neighbor in $X$. The \emph{closed neighborhood of $X$}, denoted $N[X]$, is $N(X) \cup X$. If $H$ is an induced subgraph of $G$ and $X \subseteq H$, then $N_H(X)$ ($N_H[X]$) denotes the open (closed) neighborhood of $X$ in $H$. Let $Y \subseteq V(G)$ be disjoint from $X$. Then, $X$ is \emph{anticomplete} to $Y$ if there are no edges between $X$ and $Y$. We use $X \cup v$ to mean $X \cup \{v\}$. 
 
Given a graph $G$, a {\em path in $G$} is an induced subgraph of $G$ that is a path. If $P$ is a path in $G$, we write $P = p_1 \dd \hdots \dd p_k$ to mean that $p_i$ is adjacent to $p_j$ if and only if $|i-j| = 1$. We call the vertices $p_1$ and $p_k$ the \emph{ends of $P$}, and say that $P$ is \emph{from $p_1$ to $p_k$}. The \emph{interior of $P$}, denoted by $P^*$, is the set $V(P) \setminus \{p_1, p_k\}$. The \emph{length} of a path $P$ is the number of edges in $P$. A \emph{cycle} $C$ is a sequence of vertices $p_1p_2\ldots p_kp_1$,
$k \geq 3$, such that $p_1\ldots p_k$ is a path, $p_1p_k$ is an
edge, and there are no other edges in $C$.  The {\em length} of $C$ is
the number of edges in $C$. We denote a cycle of length four by $C_4$. 

If $v \in V(G)$ and $X \subseteq V(G)$, a \emph{shortest path from $v$ to $X$} is the shortest path with one end $v$ and the other end in $X$.  If $v \in V(G)$, then $N_G^d(v)$  (or $N^d(v)$ when there is no danger of confusion) is the set of all vertices in $V(G)$ at distance exactly $d$ from $v$, and $N_G^d[v]$ (or $N^d[v]$)  is the set of vertices at distance at most $d$ from $v$. Similarly, if $X \subseteq V(G)$,  $N_G^d(X)$ (or $N^d(X)$) is the set of all vertices in $V(G)$ at distance exactly $d$ from $X$, and $N^d[X]$ (or $N^d[X]$) is the set of all vertices in $V(G)$ at distance at most $d$ from $X$. 

Next we describe a few types of graphs that we will need.  {They are illustrated in Figure \ref{fig:three-path-configs}}. 
A \emph{theta} is a graph consisting of three internally vertex-disjoint
paths $P_1 = a \dd \hdots \dd b$, $P_2 = a \dd \dots \dd b$, and
$P_3 = a \dd \dots \dd b$ of length at least 2, such that no edges exist
between the paths except the three edges incident with $a$ and the three
edges incident with $b$.
A \emph{prism} is a graph consisting of three vertex-disjoint paths
$P_1 = a_1 \dd \dots \dd b_1$, $P_2 = a_2 \dd \dots \dd b_2$, and $P_3 = a_3 \dd \dots \dd b_3$ of
length at least 1, such that $a_1a_2a_3$ and $b_1b_2b_3$ are triangles
and no edges exist between the paths except those of the two
triangles.
A \emph{pyramid} is a graph
   consisting of three paths $P_1 = a \dd \dots \dd b_1$,
   $P_2 = a \dd \dots \dd b_2$, and $P_3 = a \dd \dots \dd b_3$ of length at least 1, two
   of which have length at least 2, vertex-disjoint except at $a$, and
   such that $b_1b_2b_3$ is a triangle and no edges exist between the
   paths except those of the triangle and the three edges incident with
   $a$.

A \emph{wheel} $(H, x)$ is a hole $H$ and a vertex $x$ such that $x$ has at least three neighbors in $H$. A wheel $(H,x)$ is {\em even} if $x$ has an even number of neighbors on $H$. The following lemma characterizes odd-signable graphs in terms of forbidden induced subgraphs. 

\begin{figure}[ht]
\begin{center}
\begin{tikzpicture}[scale=0.29]

\node[label=above:{$\scriptstyle{a}$}, inner sep=2.5pt, fill=black, circle] at (0,3)(v1){}; 
\node[inner sep=2.5pt, fill=black, circle] at (3, 1)(v2){}; 
\node[inner sep=2.5pt, fill=black, circle] at (-3, 1)(v3){}; 
\node[inner sep=2.5pt, fill=black, circle] at (0, 1) (v4){};
\node[label=below:{$\scriptstyle{b}$}, inner sep=2.5pt, fill=black, circle] at (0, -7) (v5){};


\draw[black, thick] (v1) -- (v2);
\draw[black, thick] (v1) -- (v3);
\draw[black, thick] (v1) -- (v4);
\draw[black, dotted, thick] (v2) -- (v5);
\draw[black, dotted, thick] (v3) -- (v5);
\draw[black, dotted, thick] (v4) -- (v5);

\end{tikzpicture}
\hspace{0.7cm}
\begin{tikzpicture}[scale=0.29]

\node[label=above:{$\scriptstyle{a}$}, inner sep=2.5pt, fill=black, circle] at (0,3)(v1){}; 
\node[inner sep=2.5pt, fill=black, circle] at (3, 1)(v2){}; 
\node[inner sep=2.5pt, fill=black, circle] at (-3, 1)(v3){}; 
\node[label=right:{$\scriptstyle{b_1}$}, inner sep=2.5pt, fill=black, circle] at (0, -4)(v4){};
\node[label=right:{$\scriptstyle{b_2}$}, inner sep=2.5pt, fill=black, circle] at (3, -7)(v5){};
\node[label=left:{$\scriptstyle{b_3}$}, inner sep=2.5pt, fill=black, circle] at (-3, -7)(v6){}; 


\draw[black, thick] (v1) -- (v2);
\draw[black, thick] (v1) -- (v3);
\draw[black, dotted, thick] (v1) -- (v4);
\draw[black, dotted, thick] (v2) -- (v5);
\draw[black, dotted, thick] (v3) -- (v6);
\draw[black, thick] (v4) -- (v5);
\draw[black, thick] (v4) -- (v6);
\draw[black, thick] (v5) -- (v6);

\end{tikzpicture}
\hspace{0.7cm}
\begin{tikzpicture}[scale=0.29]

\node[label=above:{$\scriptstyle{a_1}$}, inner sep=2.5pt, fill=black, circle] at (0, 0)(v1){}; 
\node[label=right:{$\scriptstyle{a_2}$}, inner sep=2.5pt, fill=black, circle] at (3, 3)(v2){}; 
\node[label=left:{$\scriptstyle{a_3}$}, inner sep=2.5pt, fill=black, circle] at (-3, 3)(v3){}; 
\node[label=below:{$\scriptstyle{b_1}$}, inner sep=2.5pt, fill=black, circle] at (0, -4)(v4){}; 
\node[label=right:{$\scriptstyle{b_2}$}, inner sep=2.5pt, fill=black, circle] at (3, -7)(v5){}; 
\node[label=left:{$\scriptstyle{b_3}$}, inner sep=2.5pt, fill=black, circle] at (-3, -7)(v6){};


\draw[black, thick] (v1) -- (v2);
\draw[black, thick] (v1) -- (v3);
\draw[black, thick] (v2) -- (v3);
\draw[black, dotted, thick] (v1) -- (v4);
\draw[black, dotted, thick] (v2) -- (v5);
\draw[black, dotted, thick] (v3) -- (v6);
\draw[black, thick] (v4) -- (v5);
\draw[black, thick] (v4) -- (v6);
\draw[black, thick] (v5) -- (v6);

\end{tikzpicture}
\hspace{0.7cm}
\begin{tikzpicture}[scale=0.30]
\node[inner sep=2.5pt, fill=black, circle] at (0, 5)(v1){};
 \node[inner sep=2.5pt, fill=black, circle] at (0, 10)(v2){}; 
 \node[inner sep=2.5pt, fill=black, circle] at (-3, 1)(v3){};
 \node[inner sep=2.5pt, fill=black, circle] at (3, 1) (v4){}; 
 
 \draw[black, thick] (v1) -- (v2);
 \draw[black, thick] (v1) -- (v3);
 \draw[black, thick] (v1) -- (v4);
 \draw[black, dotted, thick] (v1) -- (2, 7);
 \draw[black, dotted, thick] (v1) -- (-2, 7);
\vspace{-0.75cm}
\draw[black, thick] (0,5) circle (5);

\end{tikzpicture}
\end{center}
\caption{  Theta, pyramid, prism, and wheel}
\label{fig:three-path-configs}
\end{figure}
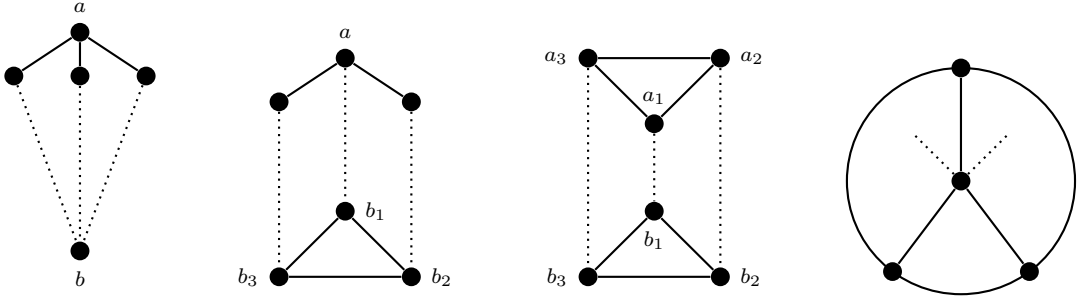

\begin{theorem}\label{os}
{\em (\cite{Conforti1999EvenGraphs})}
A graph is odd-signable if and only if it is
(even wheel, theta, prism)-free.
\end{theorem}

A \emph{cutset} $C \subseteq V(G)$ of $G$ is a set of vertices such that $G \setminus C$ is disconnected. A {\em star cutset} in a graph $G$ is a cutset $S\subseteq V(G)$ such that  either $S=\emptyset$ or for some $x\in S$, $S\subseteq N[x]$. A \emph{clique} is a set $K \subseteq V(G)$ such that every pair of vertices in $K$ are adjacent. A \emph{clique cutset} is a cutset $C \subseteq V(G)$ such that $C$ is a clique.

 \section{Balanced separators and laminar collections}
 \label{sec:laminar}

 { The goal of this section is to develop the notion of a ``central bag'' for
 a laminar collection of separations, and to study the properties of the
 central bag.   The main result  is Lemma~\ref{lemma:biglemma},
 that connects the existence of a balanced separator in the whole  graph 
with the existence of one  in the central bag of a laminar collection of separations.}  {Note that in a later paper by the authors and their coauthors \cite{wallpaper}, a simpler way to define central bags is given.}
  
For the remainder of the paper, unless otherwise specified, we assume that if $G$ is a graph, then $w:V(G) \to [0, 1]$ is a weight function of $G$ with $w(G) = 1$, and $w^{\max} = \max_{v \in V(G)} w(v)$. A \emph{separation} of a graph $G$ is a triple of disjoint vertex sets $(A, C, B)$ such that $A \cup C \cup B = V(G)$ and $A$ is anticomplete to $B$. A separation $(A, C, B)$ is \emph{proper} if $A$ and $B$ are nonempty. A set $X \subseteq V(G)$ is a \emph{clique star} if there exists a nonempty clique $K$ in $G$ such that $K \subseteq X \subseteq N[K]$. The clique $K$ is called the \emph{center} of $X$. A separation $S = (A, C, B)$ is a \emph{star separation} if $C$ is a clique star, and the \emph{center} of $S$ is the center of $C$.  For $\varepsilon \in [0, 1]$, a separation $S = (A, C, B)$ is \emph{$\varepsilon$-skewed} if $w(A) < \varepsilon$ or $w(B) < \varepsilon$. For the remainder of the paper, if $S = (A, C, B)$ is $\varepsilon$-skewed, we assume that $w(A) < \varepsilon$. Let $S_1 = (A_1, C_1, B_1)$ and $S_2 = (A_2, C_2, B_2)$ be two separations. For $i = 1, 2$, let $X_i = A_i \cup C_i$ and $Y_i = C_i \cup B_i$. We say $S_1$ and $S_2$ are {\em non-crossing} if for some $i \in \{1, 2\}$, either $X_i \subseteq X_{3-1}$ and $Y_{3-i} \subseteq Y_i$, or $X_i \subseteq Y_{3-i}$ and $X_{3-i} \subseteq Y_i$. If $S_1$ and $S_2$ are not non-crossing, then $S_1$ and $S_2$ {\em cross}.

Let $\C$ be a collection of separations of $G$. The collection $\C$ is \emph{laminar} if the separations of $\C$ are pairwise non-crossing. The \emph{separation dimension} of $\C$, denoted $\dim(\C)$, is the minimum number of laminar collections of separations with union $\C$. 

Let $G$ be a graph and let $(T, \chi)$ be a tree decomposition of $G$. Suppose that $e = t_1t_2$ is an edge of $T$ and let $T_1$ and $T_2$ be the connected components of $T \setminus e$, where for $i = 1, 2$, $t_i$ is a vertex of $T_i$. Up to symmetry between $t_1$ and $t_2$, the separation of $G$ corresponding to $e$, denoted $S_e$, is defined as follows: $S_e = (D^{t_1}_e, C_e, D^{t_2}_e)$, where $C_e = \chi(t_1) \cap \chi(t_2)$, $D^{t_1}_e = \left(\bigcup_{t \in T_1 } \chi(t)\right) \setminus C_e$, and $D^{t_2}_e = \left(\bigcup_{t \in T_2} \chi(t)\right) \setminus C_e$. The following lemma shows that given a laminar collection of separations $\C$ of $G$, there exists a tree decomposition $(T_\C, \chi_\C)$ of $G$ such that there is a bijection between $\C$ and the separations corresponding to edges of $(T_\C, \chi_\C)$.

\begin{lemma}[\cite{Robertson1991GraphTree-decomposition}]
\label{lemma:noncrossing-separations-td}
Let $G$ be a graph and let $\mathcal{C}$ be a laminar collection of separations of $G$. Then there is a tree decomposition $(T_{\mathcal{C}}, \chi_{\mathcal{C}})$ of $G$ such that 
\begin{enumerate}[(i)]
\item for all $S \in \mathcal{C}$, there exists $e \in E(T_{\mathcal{C}})$ such that $S = S_e${ , and}

\item for all $e \in E(T_{\mathcal{C}})$, $S_e \in \mathcal{C}${ .}

\end{enumerate}
\end{lemma}

We call $(T_\C, \chi_\C)$ a \emph{tree decomposition corresponding to $\C$}. Suppose $\C$ is a laminar collection of $\varepsilon$-skewed separations of $G$, and let $(T_\C, \chi_\C)$ be a tree decomposition corresponding to $\C$. For $e \in E(T_\C)$, $S_e = (A_e, C_e, B_e)$, where $w(A_e) < \varepsilon$. We define the directed tree $T'_\C$ to be the orientation of $T_\C$ given by directing edge $e = t_1t_2$ of $T_\C$ from $t_1$ to $t_2$ if $A_e = D_e^{t_1}$ (so $e = { (t_1,t_2)}$ in $T'_\C$), and from $t_2$ to $t_1$ if $A_e = D_e^{t_2}$ (so $e = { (t_2,t_1)}$ in $T'_\C)$. If $w(A_e) < \varepsilon$ and $w(B_e) < \varepsilon$, then edge $e$ is directed arbitrarily. 

A \emph{sink} of a directed graph $G$ is a vertex $v$ such that each edge incident with $v$ is oriented toward $v$. Every directed tree has at least one sink. A directed tree $T$ is an \emph{in-arborescence} if there exists a root $v \in V(T)$ such that for every $u \in V(T)$, there is exactly one directed path from $u$ to $v$ in $T$. The following lemma shows that when $\C$ is a laminar collection of $\varepsilon$-skewed separations {  satisfying an additional property}, $T'_\C$ is an in-arborescence.
 
 \begin{lemma}
 Let $\varepsilon, \varepsilon_0 > 0$ be such that $\varepsilon + \varepsilon_0 < \frac{1}{2}$. Let $G$ be a graph and let $\C$ be a laminar collection of $\varepsilon$-skewed separations of $G$ such that $w(C) \leq \varepsilon_0$ for all $(A, C, B)$ in $\C$. Let $(T_\C, \chi_\C)$ be a tree decomposition corresponding to $\C$.  Then, the directed tree $T'_\C$ is an in-arborescence. 
 \label{lemma:in-arborescence}
 \end{lemma}
 \begin{proof}
 Let $x \in V(T'_\C)$ be a sink of $T'_\C$. We prove by induction on the distance from $x$ in $T_\C$ that for every vertex $u \in V(T'_\C)$, the path from $u$ to $x$ in $T_\C$ is a directed path from $u$ to $x$ in $T'_\C$. Since $x$ is a sink, the base case follows immediately. Suppose that there is a directed path from $v$ to $x$ in $T'_\C$ for all vertices $v$ of distance $i$ from $x$, and consider vertex $u$ of distance $i+1$ from $x$. Let $P = u \dd v \dd v' \dd \hdots \dd x$ be the path from $u$ to $x$ in $T_\C$. By induction, the path $v \dd v' \dd \hdots \dd x$ is a directed path from $v$ to $x$ in $T'_\C$. Suppose that ${ (v,u)} \in E(T'_\C)$. Let $T_1$ be the component of $T'_\C \setminus { (v,u)}$ containing $v$, and let $T_2$ be the component of $T'_\C \setminus { (v,v')}$ containing $v$. Because {  $S_{vu}$ and $S_{vv'}$} are $\varepsilon$-skewed separations of $G$, we have that 
 \begin{equation} w\left(\left(\bigcup_{t \in T_1} \chi_\C(t)\right) \setminus \left(\chi_\C(v) \cap \chi_\C(u)\right)\right) < \varepsilon 
 \label{eqn:skewed_tree_1}
 \end{equation}
 and
 \begin{equation} w\left(\left(\bigcup_{t \in T_2} \chi_\C(t)\right) \setminus \left(\chi_\C(v) \cap \chi_\C(v')\right)\right) < \varepsilon. 
 \label{eqn:skewed_tree_2}
 \end{equation}
 
 Together, \eqref{eqn:skewed_tree_1} and \eqref{eqn:skewed_tree_2} imply that $w(G) < 2\varepsilon + 2\varepsilon_0 < 1$, a contradiction. Therefore, the directed tree $T'_\C$ is an in-arborescence. 
 \end{proof}

 \begin{lemma}
 \label{lemma:seps_are_skewed}
 Let $c \in [\frac{1}{2}, 1)$ and let $d$ be a positive integer. Let $G$ be a graph, let $w:V(G) \to [0, 1]$ be a weight function on $G$ with $w(G) = 1$, and suppose $G$ has no $(w, c, d)$-balanced separator. Let $S = (A, C, B)$ be a separation of $G$ such that $C$ is $d$-bounded. Then, $S$ is $(1-c)$-skewed. 
 \end{lemma}
 \begin{proof}
 Since $C$ is $d$-bounded and $G$ has no $(w, c, d)$-balanced separator, we may assume $w(B) > c$. Since $1 = w(G) \geq w(A) + w(B)$ and $w(B) > c$, it follows that $w(A) < 1 - c$, and so $S$ is $(1-c)$-skewed. 
 \end{proof}
 
 Let $G$ be a graph with maximum degree $\delta$. Note that $\delta + \delta^2$ is an upper bound for the maximum size of a clique star in $G$. Let $\beta \subseteq V(G)$. For a laminar collection $X$ of $\varepsilon$-skewed {  star separations} of $G$, $\beta$ is \emph{perpendicular} to $X$ if $\beta \cap A = \emptyset$ for all $(A, C, B) \in X$.

 \begin{lemma}
   Let $\delta$ be a positive integer, let  $ c\in [\frac{1}{2},1)$, and let $m \in [0, 1]$, with {  $(1-c)+ m(\delta + \delta^2) < \frac{1}{2}$.} Let $G$ be a connected graph with maximum degree $\delta$ and  {let $w:V(G) \to [0, 1]$ be a weight function on $G$ with $w(G) =1$ and $w^{\max} \leq m$}. Let $X$ be a laminar collection of {  $(1-c)$-skewed} star separations of $G$. Let $(T_X, \chi_X)$ be a tree decomposition corresponding to $X$.  {  (Note that since $(1-c) + w^{\max}(\delta + \delta^2) < \frac{1}{2}$, it follows from Lemma \ref{lemma:in-arborescence} that $T'_X$ is an in-arborescence.)}  Let $v$ be the root of $T'_X$ and let $\beta = \chi_X(v)$. Then $\beta$ is connected and perpendicular to $X$.
 \label{lemma:central_bag}
 \end{lemma}
 \begin{proof}
Suppose $(A, C, B) \in X$. Then, $C$ is a clique star, so $|C| \leq \delta + \delta^2$ and $w(C) \leq w^{\max} (\delta + \delta^2)$.  First, we show that $\beta$ is connected. Let $e_1, \hdots, e_m$ be the edges of $T_X$ incident with $v$ and let $S_{e_1}, \hdots, S_{e_m}$ be the corresponding separations, where $S_{e_i} = (A_{e_i}, C_{e_i}, B_{e_i})$ and {  $w(A_{e_i}) < 1-c$.} Then, $V(G) \setminus \beta = \bigcup_{i = 1}^m A_{e_i}$. 
{  Since $A_{e_1},\ldots ,A_{e_m}$ are pairwise disjoint and anticomplete,}
for every connected component $D$ of $G \setminus \beta$ there exists $1 \leq i \leq m$ such that $D \subseteq A_{e_i}$. Since $N(A_{e_i}) \cap \beta \subseteq C_{e_i}$ and {  $C_{e_i} \subseteq N[K_{e_i}]$ for some clique $K_{e_i} \subseteq C_{e_i}$}, it follows that the neighborhood in $\beta$ of every connected component of $G \setminus \beta$ is {  contained in a unique connected component of $\beta$}. Therefore, since $G$ is connected, $\beta$ is connected.

Now we show that $\beta$ is perpendicular to $X$. Let $(A, C, B) \in X$, let $e = t_1t_2$ be the edge of $T_X$ such that $S_e = (A, C, B)$, and let $T_1$ and $T_2$ be the components of $T_X \setminus e$ containing $t_1$ and $t_2$, respectively. Up to symmetry between $T_1$ and $T_2$, assume that 
{  $A = (\cup_{t\in T_1} \chi_X(t)) \setminus \chi_X(t_2)$}. Then, $e = (t_1,t_2)$ in $T'_X$. Since $v$ is the root of $T'_X$, it follows that $v \in V(T_2)$, and thus 
{  $\beta \subseteq \cup_{t\in T_2} \chi_X(t)$}. Therefore, $\beta \cap A = \emptyset$, so $\beta$ is perpendicular to $X$.  
 \end{proof}

 Let $G$ be a connected graph with maximum degree $\delta$ and let $X$ be a laminar collection of $\varepsilon$-skewed star separations of $G$, where $\varepsilon + w^{\max} (\delta + \delta^2) < \frac{1}{2}$. Let $(T_X, \chi_X)$ be a tree decomposition corresponding to $X$. Let ${  v} \in V(T_X)$  and $\beta = \chi_X({  v})$ be as in Lemma \ref{lemma:central_bag}; then $\beta$  is connected and perpendicular to $X$. We call $\beta$ the \emph{central bag} for $T_X$.  Let $e_1, \hdots, e_m$ be the edges of $T_X$ incident with ${  v}$ where $e_i = v_i{  v}$, and let $S_{e_1}, \hdots, S_{e_m}$ be the corresponding separations of $G$, where $S_{e_i} = (A_{e_i}, C_{e_i}, B_{e_i})$. Since $C_{e_i} = \chi_X({  v}) \cap \chi_X(v_i)$, it follows that $C_{e_i} \subseteq \chi_X({  v}) = \beta$ for every $i \in \{1, \hdots, m\}$. 
 
 For every $C_{e_i}$, let $K_{e_i}$ be {  a} center of $C_{e_i}$. We let $v_{e_i} \in K_{e_i}$ chosen arbitrarily be the \emph{anchor} of $C_{e_i}$. For $v \in V(G)$, let $I_v \subseteq \{1, \hdots, m\}$ be the set of indices $i$ such that $v$ is the anchor of $C_{e_i}$. Then, the \emph{weight function $w_X$ on $\beta$ with respect to $T_X$} is a function $w_X: \beta \to [0, 1]$ such that $w_X(v) = w(v) + \sum_{i \in I_v} w(A_{e_i})$ for all $v \in \beta$. 
 
 \begin{lemma}
 Let $\delta$ be a positive integer and let $\varepsilon, m \in [0, 1]$, with $\varepsilon + m (\delta + \delta^2) < \frac{1}{2}$. Let $G$ be a connected graph with maximum degree $\delta$ and  {let $w:V(G) \to [0, 1]$ be a weight function on $G$ with $w(G) = 1$ and $w^{\max} \leq m$}. Let $X$ be a laminar collection of $\varepsilon$-skewed star separations of $G$. Let $(T_X, \chi_X)$ be a tree decomposition corresponding to $X$, let $\beta$ be the central bag for $T_X$, and let $w_X$ be the weight function on $\beta$ with respect to $T_X$. Then, $w_X(\beta) = w(G) = 1$. Furthermore, if every clique $K$ of $G$ is the center of at most one star separation in $X$, then $w_X^{\max} \leq w^{\max} + 2^\delta \varepsilon$. 
 \label{lemma:weight_max}
 \end{lemma}
\begin{proof}

 By the definition of $w_X$, we have $w_X(\beta) = \sum_{v \in \beta} w_X(v) = \sum_{v \in V(G) \setminus \bigcup_{i = 1}^m A_{e_i}} w(v) + \sum_{i = 1}^m w(A_{e_i}) = w(G) = 1$. 
 
 Suppose every clique $K$ of $G$ is the center of at most one star separation in $X$. Because the maximum degree of $G$ is $\delta$, every vertex $v \in V(G)$ is in at most $2^\delta$ cliques of $G$. It follows that every vertex $v \in V(G)$ is the anchor of at most $2^\delta$ separations of $X$, so $|I_v| \leq 2^\delta$. Since $X$ is a collection of $\varepsilon$-skewed separations, $w(A_{e_i}) < \varepsilon$ for all $i \in I_v$. Therefore, $w_X^{\max} \leq w^{\max} + 2^\delta \varepsilon$. 
 \end{proof}

 The following lemma shows that if $G$ does not have a $(w, c, d)$-balanced separator and $X$ is a laminar collection of star separations of $G$, then the central bag for $T_X$ does not have a $(w_X, c, d - 2)$-balanced separator.   

 \begin{lemma}
 \label{lemma:biglemma}
Let $\delta, d$ be positive integers with $d > 2$, let $c \in [\frac{1}{2}, 1)$, and let $m \in [0, 1]$, with $(1-c) + m(\delta + \delta^2) < \frac{1}{2}$. Let $G$ be a connected graph with maximum degree $\delta$,  {let $w:V(G) \to [0, 1]$ be a weight function on $G$ with $w(G) = 1$ and $w^{\max} \leq m$}, and suppose that $G$ does not have a $(w, c, d)$-balanced separator. {  Let $X$ be a laminar collection of star separations of $G$. }
Then, the central bag $\beta$ for $X$ exists ({  in particular,} $\beta$ is perpendicular to $X$), $w_X(\beta) = 1$, and $\beta$ does not have a $(w_X, c, d - 2)$-balanced separator. 
 \end{lemma}
 \begin{proof}
   Since  $X$ is a collection of star separations, it follows that $C$ is 2-bounded for every $(A, C, B) \in X$. {  Since} $G$ does not have a $(w, c, 2)$-balanced separator Lemma \ref{lemma:seps_are_skewed} implies that  every separation in $X$ is $(1-c)$-skewed. Let $(T_X, \chi_X)$ be a tree decomposition corresponding to $X$. Then, by Lemma \ref{lemma:central_bag}, the central bag $\beta$ for $X$ exists, and by Lemma \ref{lemma:weight_max}, $w_X(\beta) = 1$. 

Suppose that $Y$ is a $(w_X, c, d-2)$-balanced separator of $\beta$. We claim that $N_\beta^{2}[Y]$ is a $(w, c, d)$-balanced separator of $G$. Since $Y$ is $(d-2)$-bounded, it follows that $N_\beta^{2}[Y]$ is $d$-bounded. Let $Q_1, \hdots, Q_\ell$ be the components of $\beta \setminus Y$. Let $t \in V(T_X)$ be such that $\beta = \chi_X(t)$. Let $e_1, \hdots, e_m$ be the edges of $T_{X}$ incident with $t$, let $S_{e_1}, \hdots, S_{e_m}$ be the corresponding separations, where $S_{e_i} = (A_{e_i}, C_{e_i}, B_{e_i})$ and $w(A_{e_i}) < 1 - c$, and let $c_{e_i}$ be the anchor of $C_{e_i}$ for $i = 1, \hdots, m$. Then, $V(G) \setminus \beta = \bigcup_{i=1}^m A_{e_i}$ and $A_{e_i}$ is anticomplete to $A_{e_j}$ for $i \neq j$. For $v \in V(G)$, let $I_v \subseteq \{1, \hdots, m\}$ be the set of all $i$ such that $v$ is the anchor of $C_{e_i}$. For $i = 1, \hdots, \ell$, let $A_i = \bigcup_{v \in Q_i} \left(\bigcup_{j \in I_v} A_{e_j}\right)$, let $Q_i' = (Q_i \setminus N_\beta^{2}[Y])$, and let $Z_i = Q_i' \cup A_i$. \\

\noindent \emph{(1) $Z_i$ is anticomplete to $Z_j$ for $i \neq j$.}

Suppose there is an edge $e$ from $Z_i$ to $Z_j$. Since $Q_i'$ is anticomplete to $Q_j'$ and $A_i$ is anticomplete to $A_j$, we may assume that $e$ is from $A_{e_{i'}}$ to $Q_j'$, where $A_{e_{i'}} \subseteq A_i$. Since $N(A_{e_{i'}}) \cap \beta \subseteq C_{e_{i'}}$, it follows that $C_{e_{i'}} \cap Q_j' \neq \emptyset$. Let $v \in C_{e_{i'}} \cap Q_j'$ and let $P$ be a shortest path from $c_{e_{i'}}$ to $v$ through $\beta$. Since $c_{e_{i'}}, v \in C_{e_{i'}}$ and $C_{e_{i'}}$ is a clique star, it follows that $P$ is of length at most 2. Since $c_{e_{i'}} \in Q_i$ and $v \in Q_j$, it follows that $P$ goes through $Y$ and thus $P$ is of length exactly 2. 
Let $P = c_{e_{i'}}\dd {  y} \dd v$, where ${  y} \in Y$. Then, $v \in N_\beta^{2}[{  y}] \subseteq N_\beta^{2}[Y]$, a contradiction 
{  (since $v\in Q_j'$)}. This proves (1). \\

\noindent \emph{(2) If $c_{e_i} \in Y$, then $A_{e_i}$ is anticomplete to $Z_j$ for $j \in \{1, \hdots, \ell\}$.}

Suppose $c_{e_i} \in Y$. Then, $C_{e_i} \subseteq N_\beta^{2}[Y]$. Since $N(A_{e_i}) \cap \beta \subseteq C_{e_i}$, it follows that $A_{e_i}$ is anticomplete to $Q_j'$ for all $j = 1, \hdots, \ell$. Therefore, $A_{e_i}$ is anticomplete to $Z_j$ for all $j = 1, \hdots, \ell$. This proves (2). \\

Let $I_Y \subseteq \{1, \hdots, m\}$ be the set of all $i$ such that $c_{e_i} \in Y$. Then, $V(G) \setminus N_\beta^{2}[Y] = \left(\bigcup_{i \in I_Y} A_{e_i}\right) \cup \left(\bigcup_{j=1}^\ell Z_j\right)$. Suppose $Z$ is a component of $V(G) \setminus N_\beta^{2}[Y]$. It follows from (1) and (2) that either $Z \subseteq A_{e_i}$ for some $i \in I_Y$, or $Z \subseteq Z_j$ for some $j \in \{1, \hdots, \ell\}$. Since $w_X(Q_i) \leq c$, it follows that $w(Z_i) \leq c$ for all $i = 1, \hdots, \ell$.  Further, since every separation in $X$ is $(1-c)$-skewed and $c \in [\frac{1}{2}, 1)$, it follows that $w(A_{e_i}) < (1-c) \leq c$ for all $i \in I_Y$. Therefore, $w(Z) \leq c$, and $N_\beta^{2}[Y]$ is a $(w, c, d)$-balanced separator of $G$, a contradiction.  
 \end{proof}

 \section{Balanced separators and clique separations}
 \label{sec:clique}

In this section, we show that if $G$ is a connected graph with no balanced separator, then there exists a connected induced subgraph of $G$ with no balanced separator and no clique cutset. The central bag from Lemma \ref{lemma:biglemma} is the primary tool for finding such an induced subgraph. 

A separation $(A, C, B)$ of a graph $G$ is a \emph{clique separation} if $C$ is a clique. A clique cutset $C$ is \emph{minimal} if every $c \in C$ has a neighbor in every component of $G \setminus C$. Note that in a connected graph $G$, $|C| \geq 1$ for every minimal clique cutset $C$ of $G$. 
\begin{lemma} 
 Let $G$ be a connected graph and let $\C$ be a collection of clique separations of $G$ such that $C$ is a minimal clique cutset for all $(A, C, B) \in \C$ and for every two distinct separations $(A_1, C_1, B_1)$, $(A_2, C_2, B_2) \in \C$, $C_1 \neq C_2$. Then, $\dim(\C) = 1$. In particular, $\C$ is laminar.
 \label{lemma:clique_cutsets_laminar}
\end{lemma}
\begin{proof}
Let $S_1 = (A_1, C_1, B_1)$ and $S_2 = (A_2, C_2, B_2)$ be clique separations of $G$ such that $C_1$ and $C_2$ are minimal clique cutsets of $G$. Since $C_1$ is a clique and $A_2$ is anticomplete to $B_2$, either $C_1 \cap A_2 = \emptyset$ or $C_1 \cap B_2 = \emptyset$. We may assume that $C_1 \cap A_2 = \emptyset$. Similarly, we may assume that $C_2 \cap A_1 = \emptyset$. If $A_1 \cap A_2 = \emptyset$, then $A_2 \subseteq B_1$ and $A_1 \subseteq B_2$, so $S_1$ and $S_2$ are non-crossing (since $A_2 \cup C_2 \subseteq B_1 \cup C_1$ and $A_1 \cup C_1 \subseteq B_2 \cup C_2)$. Therefore, we may assume that $A_1 \cap A_2 \neq \emptyset$. Since $C_1 \neq C_2$, either $C_1 \cap B_2 \neq \emptyset$ or $C_2 \cap B_1 \neq \emptyset$. Assume up to symmetry that $C_1 \cap B_2 \neq \emptyset$. Since $A_1 \subseteq A_2 \cup B_2$ and $A_2$ is anticomplete to $B_2$, every component of $A_1$ is either a subset of $A_2$ or a subset of $B_2$. 
{  Since $A_1\cap A_2\neq \emptyset$, there exists
a connected component $A$ of $A_1$ such that $A \subseteq A_2$. Let $c \in C_1 \cap B_2$.} Then, $c$ is anticomplete to $A$, contradicting that $C_1$ is a minimal clique cutset. It follows that $S_1$ and $S_2$ are non-crossing. Therefore, $\dim(\C) = 1$. 
\end{proof}

Let $G$ be a graph and let $C$ be a minimal clique cutset of $G$.
 The \emph{minimal clique separation $S$ for $C$} is defined as follows: $S = (A, C, B)$, where $B$ is a largest weight connected component of $G \setminus C$ and $A = V(G) \setminus (B \cup C)$. 
 \begin{lemma}
 Let $c \in [\frac{1}{2}, 1)$. Let $G$ be a graph, let $w:V(G) \to [0, 1]$ be a weight function on $G$ with $w(G) = 1$, and suppose $G$ has no $(w, c, 1)$-balanced separator. Let $C$ be a minimal clique cutset of $G$. Then, the minimal clique separation $S$ for $C$ is unique and $S$ is $(1-c)$-skewed. 
 \label{lemma:clique_skewed}
 \end{lemma}
 \begin{proof}
 Since $G$ has no $(w, c, 1)$-balanced separator, $C$ is not a $(w, c, 1)$-balanced separator. It follows that if $B$ is a largest weight connected component of $G \setminus C$, then $w(B)> c$. Since $c \in [\frac{1}{2}, 1)$ and $w(G) = 1$, the largest weight connected component of $G \setminus C$ is unique, and thus $S$ is unique. Since $C$ is a 1-bounded set and $G$ has no $(w, c, 1)$-balanced separator, it follows from Lemma \ref{lemma:seps_are_skewed} that $S$ is $(1-c)$-skewed. 
 \end{proof}

In the following lemma, we prove that if $k$ is the minimum size of a  clique cutset in $G$ and $\C$ is the collection of all minimal clique separations of $G$ 
{  for clique cutsets} of size $k$, then the central bag $\beta$ for $\C$ does not contain a clique cutset of size less than or equal to $k$. Note that a minimum size clique cutset is a minimal clique cutset. 

\begin{lemma}
  \label{lemma:clique_k}
  Let $\delta$ be a positive integer, let $k$ be a nonnegative integer, let $c \in [\frac{1}{2}, 1)$, and let $m \in [0, 1]$, with $(1-c) + m(\delta + \delta^2) < \frac{1}{2}$. Let $G$ be a connected graph with maximum degree $\delta$ and  {let $w:V(G) \to [0, 1]$ be a weight function on $G$ with $w(G) = 1$ and $w^{\max} \leq m$}. Suppose $G$ does not have a $(w, c, 1)$-balanced separator, and suppose the smallest clique cutset in $G$ has size $k$. Let $\C$ be the collection of all minimal clique separations of $G$ such that $|C| = k$ for every $(A, C, B) \in \C$. Then, $\C$ is laminar, and if $(T_\C, \chi_\C)$ is the tree decomposition of $G$ corresponding to $\C$ and $\beta$ is the central bag for $T_\C$, then $\beta$ does not have a clique cutset of size less than or equal to $k$.
\end{lemma}
\begin{proof}
Since $G$ is connected, $k \geq 1$. Since $G$ does not have a $(w, c, 1)$-balanced separator and $c \in [\frac{1}{2}, 1)$, it follows that every minimal clique cutset of size $k$ in $G$ corresponds to exactly one minimal clique separation in $\C$. Therefore, by Lemma \ref{lemma:clique_cutsets_laminar}, $\C$ is laminar, and by Lemma \ref{lemma:clique_skewed}, every separation in $\C$ is $(1-c)$-skewed. Let $v \in V(T_\C)$ be such that $\beta = \chi_\C(v)$ is the central bag for $T_\C$, and suppose $\beta$ has a clique cutset of size less than or equal to $k$. Let $(A_v, C_v, B_v)$ be a minimal clique separation of $\beta$ such that $|C_v| \leq k$. Let $v_1, \hdots, v_m$ be the vertices of $T_\C$ adjacent to $v$, let $e_i = vv_i$ be the edge from $v$ to $v_i$ for $i = 1, \hdots, m$, and let $S_{e_1}, \hdots, S_{e_m}$ be the clique separations corresponding to $e_1, \hdots, e_m$, where $S_{e_i} = (D_{e_i}^v, C_{e_i}, D_{e_i}^{v_i})$ as in Section \ref{sec:laminar}. Since $\beta \cap \chi_\C(v_i) = C_{e_i}$ and $C_{e_i}$ is a clique, it follows that $C_{e_i} \cap A_v = \emptyset$ or $C_{e_i} \cap B_v = \emptyset$ for all $i = 1, \hdots, m$. Let $A$ be the union of $A_v$ and all $D_{e_i}^{v_i}$ for $i$ such that $C_{e_i} \cap B_v = \emptyset$, and let $B$ be the union of $B_v$ and all $D_{e_i}^{v_i}$ for $i$ such that $D_{e_i}^{v_i} \not \subseteq A$. For $i \neq j$, $D_{e_i}^{v_i}$ and $D_{e_j}^{v_j}$ are disjoint and anticomplete to each other. By properties of the tree decomposition, $\beta \cup \bigcup_{i = 1}^m D_{e_i}^{v_i} = V(G)$. Therefore, it follows that $(A, C_v, B)$ is a clique separation of $G$ with $|C_v| \leq k$. 
 
 Since the smallest clique cutset in $G$ has size $k$, it follows that $|C_v| = k$. Let $S = (X, C_v, Y)$ be the minimal clique separation for $C_v$ in $G$. It follows that $S \in \C$, so by Lemma \ref{lemma:central_bag}, $\beta \subseteq C_v \cup Y$. But since $(A, C_v, B)$ is a clique separation of $G$, it follows that two components of $G \setminus C_v$ intersect $\beta$, a contradiction. 
\end{proof}

In the following theorem, we use Lemmas \ref{lemma:biglemma} and \ref{lemma:clique_k} to find an induced subgraph of $G$ that has no clique cutset and no balanced separator.  

\begin{theorem}
\label{thm:clique_free_bag}
{  Let $\delta, d$ be positive integers, 
 with $d > 2\delta - 2$.} Let $c \in [\frac{1}{2}, 1)$ and let $m \in [0, 1]$, with $(1-c) + [m + (\delta-1) 2^{\delta}(1-c)](\delta + \delta^2) < \frac{1}{2}$. Let $G$ be a connected graph with maximum degree $\delta$,  {let $w:V(G) \to [0, 1]$ be a weight function on $G$ with $w(G) = 1$ and $w^{\max} \leq m$}, and suppose $G$ has no $(w, c, d)$-balanced separator. Then, there exists a sequence $(\alpha_0, w_0), (\alpha_1, w_1), \hdots, (\alpha_{\delta'}, w_{\delta'})$ such that $\delta' < \delta$, $(\alpha_0, w_0) = (G, w)$ and for $i \in \{0, \hdots, \delta'\}$, the following hold: 
\begin{itemize}
\itemsep -0.2em
    \item $\alpha_i$ is a connected induced subgraph of $G$ and $w_i$ is a weight function on $\alpha_i$ such that $w_i(\alpha_i) = 1$ and $w_i^{\max} \leq w^{\max} + i2^{\delta}(1-c)$. 
    
    \item $\alpha_i$ has no $(w_i, c, d-2i)$-balanced separator.
    
    \item If $i > 0$ then $\alpha_i$ is the central bag for a tree decomposition corresponding to a collection of minimal clique separations of $\alpha_{i-1}$.
    
    \item $\alpha_{\delta'}$ does not have a clique cutset.
\end{itemize}
\end{theorem}

\begin{proof}
We may assume that $G$ {  has} a clique cutset, otherwise the result holds with $\delta' = 0$. If $\delta = 1$, then $G$ consists of a single edge,
{  contradicting the assumption that $G$ has a clique cutset.}
 Therefore, $\delta \geq 2$ and so $d > 2$. Since the maximum degree of $G$ is $\delta$ and every vertex in a minimal clique cutset $C$ has a neighbor in every component of $G \setminus C$, it follows that every minimal clique cutset of $G$ has size at most $\delta - 1$. Let $j_0$ be the size of the smallest clique cutset of $G$. Note that since $G$ is connected, $j_0 \geq 1$. Since $G$ has no $(w, c, d)$-balanced separator and $d \geq 1$, $G$ has no $(w, c, 1)$-balanced separator. Let $\C_1$ be the collection of all minimal clique separations of $G$ that correspond to clique cutsets of size $j_0$. By Lemma \ref{lemma:clique_skewed}, 
 {  every separation in $\C_1$ is $(1-c)$-skewed and}
 for every two distinct separations $(A_1, C_1, B_1), (A_2, C_2, B_2) \in \C_1$, $C_1 \neq C_2$. Therefore, by Lemma \ref{lemma:clique_cutsets_laminar}, $\C_1$ is laminar. Let $(T_{\C_1}, \chi_{\C_1})$ be the tree decomposition of $G$ corresponding to $\C_1$. 
 By Lemma \ref{lemma:biglemma}, the central bag for $T_{\C_1}$ exists and does not have a $(w_{\C_1}, c, d-2)$-balanced separator. Let $\alpha_1$ be the central bag for $T_{\C_1}$ and let $w_1 = w_{\C_1}$. 
 By Lemma \ref{lemma:weight_max}, $w_1(\alpha_1) = 1$ and $w_1^{\max} \leq w^{\max} + 2^{\delta}(1-c)$. 
 {  Since $(1-c) + w^{\max}(\delta + \delta^2) \leq (1-c) + [w^{\max} + (\delta - 1)2^{\delta}(1-c)](\delta + \delta^2) < \frac{1}{2}$, by Lemma \ref{lemma:central_bag}, 
 $\alpha_1$ is connected.} It follows from Lemma \ref{lemma:clique_k} that $\alpha_1$ does not have a clique cutset of size less than or equal to $j_0$.  If $\alpha_1$ does not have a clique cutset, then $\delta' = 1$ and the sequence ends. Otherwise, for $i \in \{2, \hdots, \delta - 1\}$, we define $(\alpha_i, w_i)$ inductively. For $i \in \{2, \hdots, \delta-1\}$, suppose $(\alpha_{i-1}, w_{i-1})$ are such that $\alpha_{i-1}$ is the central bag for a tree decomposition corresponding to a collection of minimal clique separations of $\alpha_{i-2}$ and $w_{i-1}$ is the corresponding weight function on $\alpha_{i-1}$, $\alpha_{i-1}$ is a connected induced subgraph of $G$ with no $(w_{i-1}, c, d_{i-1})$-balanced separator for $d_{i-1}= d - 2(i-1)$, $w_{i-1}(\alpha_{i-1}) = 1$, and $w^{\max}_{i-1} \leq w^{\max} + (i-1)2^\delta (1-c)$. Further, suppose the smallest clique cutset in $\alpha_{i-1}$ has size $j_{i-1}$, where $\delta > j_{i-1} \geq i$. 

{  Since $\delta >i$ and $d>2\delta -2$, it follows that $d-2(i-1)\geq 1$.}
Since $\alpha_{i-1}$ has no $(w_{i-1}, c, d-2(i-1))$-balanced separator, it follows that $\alpha_{i-1}$ has no $(w_{i-1}, c, 1)$-balanced separator. Let $\C_i$ be the collection of all minimal clique separations of $\alpha_{i-1}$ that correspond to clique cutsets of size $j_{i-1}$. By Lemmas \ref{lemma:clique_skewed} and \ref{lemma:clique_cutsets_laminar}, $\C_i$ is laminar. 
Since $w_{i-1}^{\max} \leq w^{\max} + (i-1)2^{\delta}(1-c)$ and $i < \delta$, it follows that $(1-c) + w_{i-1}^{\max}(\delta + \delta^2) < (1-c) + [w^{\max} + (\delta-1) 2^{\delta}(1-c)](\delta + \delta^2) < \frac{1}{2}$. 
Since $d > 2\delta - 2$, {   $i < \delta$, and $\delta \geq 2$,} it follows that $d_{i-1} = d - 2(i-1) \geq d - 2(\delta - 2) > 2$. 
Since $\alpha_{i-1}$ has no $(w_{i-1}, c, d-2(i-1))$-balanced separator, $d_{i-1} > 2$ and $(1-c) + w^{\max}_{i-1}(\delta + \delta^2) < \frac{1}{2}$, it follows from Lemma \ref{lemma:biglemma} that the central bag for $\C_i$ exists and does not have a $(w_{\C_i}, c, d_i)$-balanced separator, where $d_i = d_{i-1} -2  = d - 2i \geq 1$. Let $T_{\C_i}$ be the tree decomposition of $\alpha_{i-1}$ corresponding to $\C_i$. Let $\alpha_i$ be the central bag for $T_{\C_i}$ and let $w_i = w_{\C_i}$ be the weight function on $\alpha_i$ with respect to $T_{\C_i}$.  
By Lemma \ref{lemma:weight_max}, $w_i(\alpha_i) = 1$ and $w_i^{\max} \leq w_{i-1}^{\max} + 2^\delta (1 - c) \leq w^{\max} + i2^{\delta}(1-c)$. 
{  Since $(1-c) + w_{i-1}^{\max}(\delta + \delta^2) < \frac{1}{2}$,  by Lemma \ref{lemma:central_bag}, $\alpha_i$ is connected.} If $\alpha_i$ has no clique cutset, then $\delta' = i$ and the sequence ends. Otherwise, let $j_i$ be the size of the smallest clique cutset in $\alpha_i$. By Lemma \ref{lemma:clique_k}, it follows that $j_i > j_{i-1}$, so $j_i \geq i+1$. Since the maximum size of a minimal clique cutset in $G$, and thus in $\alpha_i$, is $\delta - 1$, $j_i < \delta$. Thus, minimal clique cutsets used in this proof are of sizes in $\{1, \hdots, \delta-1\}$, so $\delta' < \delta$. Therefore, the sequence $(\alpha_1, w_1), \hdots, (\alpha_{\delta'}, w_{\delta'})$ is well-defined and satisfies the theorem. Further, by construction, $\alpha_{\delta'}$ does not have a clique cutset.
\end{proof}

We call $\alpha_{\delta'}$ the \emph{clique-free bag} for $G$.

\section{Star cutsets and forcers}\label{sec:forcers}

Let $G$ be a graph. A cutset $C$ of $G$ is a \emph{clique star cutset} of $G$ if $C$ is a clique star. Recall that a {  star separation $S = (A, C, B)$ is proper  if $C$ is a clique star cutset.}  {In this section we study properties of separations associated with clique star cutsets. In particular, we establish the notion of a canonical separation that corresponds to a given clique, and
show how to partition a set of canonical clique separations into a bounded number of laminar collections; this is done in Lemma~\ref{lemma:csc_centers_partition}. Then we list several lemmas showing that certains subgraphs are clique star cutset forcers (Lemmas~\ref{lemma:proper_wheel_forcer}, \ref{lemma:univ_wheel_forcer}, and  \ref{lemma:short_pyramid_forcer}, summarized in Lemma~\ref{lemma:forcer_starcutset}).  Finally we show that repeatedly taking central bags leads to a forcer-free subgraph (this is done in Theorem~\ref{thm:strong_forcers_clean}).}

In the following lemma, we show that if two proper star separations cross, then their centers are not anticomplete to each other. 

\begin{lemma}
Let $G$ be a theta-free graph with no clique cutset, let $K_1$ and $K_2$ be cliques of $G$, and let $\mathcal{S}_{1} = (A_{1}, C_{1}, B_{1})$ and $\mathcal{S}_{2} = (A_{2}, C_{2}, B_{2})$ be proper star separations such that $C_1 \subseteq N[K_1]$ and $C_2 \subseteq N[K_2]$. Suppose $S_{1}$ and $S_{2}$ cross. Then, $K_1$ and $K_2$ are not anticomplete to each other. 
\label{lemma:star_sepns_cross}
\end{lemma}
\begin{proof}
Suppose $K_1$ is anticomplete to $K_2$. Then, $K_1 \cap N[K_2] = \emptyset$, so $K_1$ is contained in a connected component of $G \setminus C_{2}$. Similarly,  $K_2$ is contained in a connected component of $G \setminus C_{1}$. Up to symmetry between $A$ and $B$, assume that $K_1 \subseteq B_{2}$ and $K_2 \subseteq B_{1}$. Then, $C_{1} \cap A_{2} = \emptyset$ and $C_{2} \cap A_{1} = \emptyset$. Since $S_{1}$ and $S_{2}$ cross, it follows that $A_{1} \cap A_{2} \neq \emptyset$. Let $A = A_1 \cap A_2$. Suppose $C_{1} \subseteq B_{2}$. Then, $C_{1}$ is anticomplete to $A$. Because $A \subseteq A_{1}$ and $A_{1}$ is anticomplete to $B_{1}$, it follows that $B_{1}$ is anticomplete to $A$. Finally, since $A_{1} \cap C_{2} = \emptyset$, it follows that $A_{1} \setminus A \subseteq B_{2}$, so $A$ is anticomplete to $A_{1} \setminus A$. Therefore, $A$ is anticomplete to $G \setminus A$, a contradiction, so $C_{1} \cap C_{2} \neq \emptyset$. 

Let $C = C_{1} \cap C_{2}$, let $A'$ be a connected component of $A$, and let $C' = N_C(A')$. Suppose there exists $c_1, c_2 \in C'$ such that $c_1c_2 \not \in E(G)$. Then, $G$ contains a theta between $c_1$ and $c_2$ through $A'$, $K_1$, and $K_2$, a contradiction. Therefore, $C'$ is a clique. Since $A_{1} \cap A_{2}$ is anticomplete to $B_{1}$ and $B_{2}$, it follows that $N(A) \subseteq C$, so $N(A') = C'$. Then, $A'$ is a connected component of $G \setminus C'$, so $C'$ is a clique cutset of $G$, a contradiction.  
\end{proof}

The next lemma shows that if $Y$ is a set of cliques of size at most $k$, then there exists a partition of $Y$ into $(k + \delta k)\sum_{j=0}^{k-1} {\delta \choose j} + 1$ parts such that every two cliques in the same part are anticomplete to each other. 

\begin{lemma}
Let $\delta, k$ be positive integers with $k \leq \delta$ and let $f(k, \delta) = (k + \delta k)\sum_{j=0}^{k-1} {\delta \choose j} + 1$. Let $G$ be a graph with maximum degree $\delta$ and let $Y = \{K_1, \hdots, K_t\}$ be a set of cliques of $G$ of size at most $k$. Then, there exists a partition $(Y_1, \hdots, Y_{f(k, \delta)})$ of $Y$ such that for every $\ell \in \{1, \hdots, f(k, \delta)\}$ and $K_i, K_j \in Y_\ell$, $K_i$ is anticomplete to $K_j$.
\label{lemma:csc_centers_partition}
\end{lemma}
\begin{proof}
Let $H$ be a graph with vertex set $V(H) = \{x_1, \hdots, x_t\}$, and for $x_i, x_j \in V(H)$, {  $i\neq j$,} let $x_ix_j \in E(H)$  if and only if $K_i$ is not anticomplete to $K_j$ in $G$. Let $x_i \in V(H)$ and let {   $x_j \in N_H(x_i)$.}Then, $K_i$ is not anticomplete to $K_j$, so 
$K_j \cap N[K_i] \neq \emptyset$. Let $v \in K_j \cap N[K_i]$. Then, $K_j \subseteq N[v]$. 
Since {  $|N[K_i]| \leq k + \delta k$} and $|N[{  u}]| \leq \delta$ for all ${  u} \in V(G)$, it follows that $K_i$ is not anticomplete to at most $(k + \delta k) \sum_{j=0}^{k-1} {\delta \choose j}$ cliques of size at most $k$. Therefore, the maximum degree of $H$ is at most $(k + \delta k) \sum_{j=0}^{k-1} {\delta \choose j}$. 

Since the maximum degree of $H$ is at most $(k + \delta k) \sum_{j=0}^{k-1} {\delta \choose j}$, it follows that {  the chromatic number of $H$ is at most} $(k + \delta k) \sum_{j=0}^{k-1} {\delta \choose j} + 1 = f(k, \delta)$. Let $C: V(H) \to \{1, \hdots, f(k, \delta)\}$ be a coloring of $H$ and let $Y_1, \hdots, Y_{f(k, \delta)}$ be the color classes of $C$. Then, $(Y_1, \hdots, Y_{f(k, \delta)})$ is a partition of $Y$ such that if $\ell \in \{1, \hdots, f(k, \delta)\}$ and $K_i, K_j \in Y_\ell$, then $K_i$ is anticomplete to $K_j$.  
\end{proof}

Let $G$ be a graph with weight function $w$ and let $K$ be a nonempty clique of $G$. A {\em canonical star separation for $K$}, denoted $S_K$, is defined as follows: $S_K = (A_K, C_K, B_K)$, where $B_K$ is a largest weight connected component of $G \setminus N[K]$ if $G \setminus N[K] \neq \emptyset$ and $B_K = \emptyset$ otherwise, $C_K$ is the union of $K$ and {  the set of all vertices} $v \in N[K]$ such that $v$ has a neighbor in $B_K$, and $A_K = V(G) \setminus (B_K \cup C_K)$. The following lemma shows that if $G$ has no balanced separator, then the canonical star separation is unique. 

\begin{lemma}
\label{lemma:skewed}
Let $c \in [\frac{1}{2}, 1)$. Let $G$ be a graph with no $(w, c, 2)$-balanced separator and let $K$ be a nonempty clique of $G$. Then, the canonical star separation $S_K$ for $K$ is unique and $S_K$ is $(1-c)$-skewed. 
\end{lemma}
\begin{proof}
Since $G$ has no $(w, c, 2)$-balanced separator, $N[K]$ is not a $(w, c, 2)$-balanced separator. It follows that if $B_K$ is a largest weight connected component of $G \setminus N[K]$, then $w(B_K) > c$. Since $c \in [\frac{1}{2}, 1)$ and $w(G) = 1$, the largest weight connected component of $G \setminus N[K]$ is unique, and thus $S_K$ is unique. Since $C_K$ is a 2-bounded set and $G$ has no $(w, c, 2)$-balanced separator, it follows from Lemma \ref{lemma:seps_are_skewed} that $S_K$ is $(1-c)$-skewed. 
\end{proof}

Let $G$ be a graph. Let $X, Y, Z$ be disjoint subsets of $V(G)$. We say that
{\em $X$ separates $Y$ from $Z$} if there exist distinct components $C_Y,C_Z$
of $G\setminus X$ such that $Y \subseteq C_Y$ and $Z \subseteq C_Z$. Recall that a \emph{wheel} $(H, x)$ of $G$ consists of a hole $H$ and a vertex $x$ that has at least three neighbors in $H$. A \emph{sector} of $(H,x)$ is a path $P$ of $H$ whose ends are adjacent to $x$, and such that $x$ is anticomplete to $P^*$ {  (recall that $P^*$ is the set of interior vertices of $P$)}. A sector $P$ is a \emph{long sector} if $P^*$ is nonempty.  We now define several types of wheels that we will need.  {They are illustrated in Figure \ref{fig:types-of-wheels}}. 

A wheel $(H, x)$
is a \emph{universal
wheel} if $x$ is complete to $H$. A wheel $(H, x)$ is a \emph{twin wheel} if $N(x) \cap H$ induces a path of length 2. If $(H, x)$ is a twin wheel and $x_1 \dd x_2 \dd x_3$ is the path of length 2 induced by $N(x) \cap H$, we say $x_2$ is the \emph{clone of $x$ in $H$}. Note that if $(H, x)$ is a twin wheel and $x_2$ is the clone of $x$ in $H$, then $((H \setminus \{x_2\}) \cup \{x\}, x_2)$ is also a twin wheel. 
Suppose $(H, x)$ is a twin wheel {  contained in a graph $G$} and $x_2$ is the clone of $x$ in $H$. We say $(H, x)$ is \emph{$x$-rich} if there is a path {  in $G$} from $x$ to $V(H) \setminus N[x]$ containing no neighbors of $x_2$ other than $x$, and \emph{$x_2$-rich} if there is a path {  in $G$} from $x_2$ to $V(H) \setminus N[x]$ containing no neighbors of $x$ other than $x_2$. We say $(H, x)$ is \emph{$x$-poor} if it is not $x$-rich, and \emph{$x_2$-poor} if it is not $x_2$-rich. We say that $(H, x, x_2)$ is a \emph{terminal twin wheel} if $(H, x)$ is a twin wheel and $x_2$ is the clone of $x$ in $H$, and $(H, x)$ is either $x$-poor or $x_2$-poor. A wheel $(H, x)$ is a \emph{short pyramid} if $|N(x) \cap H| = 3$ and $x$ has exactly {  one pair of} adjacent neighbors in $H$. A wheel is \emph{proper} if it is not a twin wheel or a short pyramid. If $(H, x)$ is a short pyramid ({  resp.} proper wheel), then $x$ is said to be the \emph{center} of a short pyramid ({  resp.} proper wheel) in $H$. 

\begin{figure}[ht]
\begin{center}
\begin{tikzpicture}[scale=0.29]

\node[inner sep=2.5pt, fill=black, circle] at (0, 5)(v1){};
 \node[inner sep=2.5pt, fill=black, circle] at (0, 10)(v2){}; 
 \node[inner sep=2.5pt, fill=black, circle] at (0, 0)(v3){};
 \node[inner sep=2.5pt, fill=black, circle] at (5, 5) (v4){}; 
  \node[inner sep=2.5pt, fill=black, circle] at (-5, 5) (v5){}; 
  \node[inner sep=2.5pt, fill=black, circle] at (3.54, 8.54) (v6){}; 
\node[inner sep=2.5pt, fill=black, circle] at (-3.54, 8.54) (v7){}; 
  \node[inner sep=2.5pt, fill=black, circle] at (3.54, 1.46) (v8){}; 
  \node[inner sep=2.5pt, fill=black, circle] at (-3.54, 1.46) (v9){};

 \draw[black, thick] (v1) -- (v2);
 \draw[black, thick] (v1) -- (v3);
 \draw[black, thick] (v1) -- (v4);
 \draw[black, thick] (v1) -- (v5);
  \draw[black, thick] (v1) -- (v6);
 \draw[black, thick] (v1) -- (v7);
 \draw[black, thick] (v1) -- (v8);
 \draw[black, thick] (v1) -- (v9);

\vspace{-0.75cm}
\draw[black, thick] (0,5) circle (5);

\end{tikzpicture}
\hspace{0.7cm}
\begin{tikzpicture}[scale=0.29]

\node[inner sep=2.5pt, fill=black, circle] at (0, 5)(v1){};
 \node[inner sep=2.5pt, fill=black, circle] at (0, 0)(v2){}; 
 \node[inner sep=2.5pt, fill=black, circle] at (-3, 1)(v3){};
 \node[inner sep=2.5pt, fill=black, circle] at (3, 1) (v4){}; 
 
 \draw[black, thick] (v1) -- (v2);
 \draw[black, thick] (v1) -- (v3);
 \draw[black, thick] (v1) -- (v4);
\vspace{-0.75cm}
\draw[black, thick] (0,5) circle (5);

\end{tikzpicture}
\hspace{0.7cm}
\begin{tikzpicture}[scale=0.30]
\node[inner sep=2.5pt, fill=black, circle] at (0, 5)(v1){};
 \node[inner sep=2.5pt, fill=black, circle] at (0, 10)(v2){}; 
 \node[inner sep=2.5pt, fill=black, circle] at (-2, 0.4)(v3){};
 \node[inner sep=2.5pt, fill=black, circle] at (2, 0.4) (v4){}; 
 
 \draw[black, thick] (v1) -- (v2);
 \draw[black, thick] (v1) -- (v3);
 \draw[black, thick] (v1) -- (v4);
\vspace{-0.75cm}
\draw[black, thick] (0,5) circle (5);

\end{tikzpicture}
\end{center}
\caption{  Universal wheel, twin wheel, and short pyramid}
\label{fig:types-of-wheels}
\end{figure}
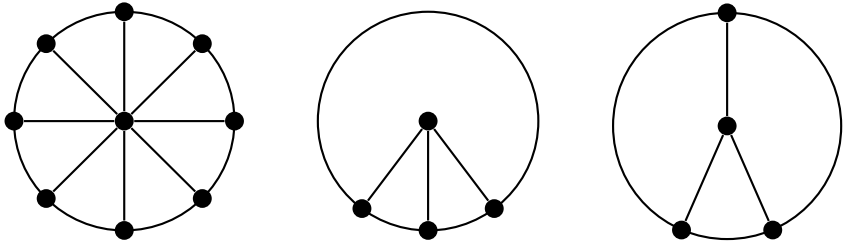

The following three lemmas show that proper wheels and short pyramids generate clique star cutsets. 

\begin{lemma}[\cite{Addario-Berry2008BisimplicialGraphs}, \cite{daSilva2013Decomposition2-joins}]
Let $G$ be a $C_4$-free odd-signable graph that contains a proper wheel $(H, x)$ that is not a universal wheel. Let $x_1$ and $x_2$ be the endpoints of a long sector $Q$ of $(H, x)$. Let $W$ be the set of all vertices $h$ in $H \cap N(x)$
such that the subpath of $H \setminus \{x_1\}$ from $x_2$ to $h$ contains an even number of neighbors of $x$, and let $Z = H \setminus (Q \cup N(x))$. Let $N' = N(x)\setminus W$. Then, $N' \cup \{x\}$ is a cutset of $G$ that separates $Q^*$ from $W \cup Z$.
\label{lemma:proper_wheel_forcer}
\end{lemma}

\begin{lemma}[\cite{daSilva2007TriangulatedGraphs}]
\label{lemma:univ_wheel_forcer}
Let $G$ be a $C_4$-free odd-signable graph that contains a universal wheel $(H,x)$.
If $G=N[x]$ then for every two non-adjacent vertices $a$ and $b$ of $H$, $N[x]\setminus \{ a,b\}$ is a cutset of $G$ that separates $a$ and $b$.
If $G\setminus N[x]\neq \emptyset$ then for every connected component $C$ of $G\setminus N[x]$, there exists $a\in H$ such that $a$
has no neighbor in $H$, i.e. $N[x]\setminus \{ a\}$ is a cutset of $G$ that separates $a$ from $C$.
\end{lemma}

\begin{lemma}
{\em (\cite{Conforti2002Even-hole-freeTheorem})}
\label{lemma:short_pyramid_forcer}
Let $G$ be a $C_4$-free odd-signable graph that contains a wheel $(H,x)$ that is a short pyramid. Let $x_1,x_2$ and $y$ be the neighbors of $x$ in $H$
such that $x_1x_2$ is an edge. For $i \in \{1,2\}$ let $H_i$ be the  sector of $(H,x)$  with ends $y,x_i$.
Then, $H_1$ and $H_2$ are long sectors of $(H, x)$, and $S = N(x) \cup N(y)$ is a cutset of $G$ that separates $H_1 \setminus S$
from $H_2 \setminus S$.
\end{lemma}

Let $G$ be a graph. A \emph{forcer} $F = (H, K)$ in $G$ consists of a hole $H$ and a clique $K$ such that one of the following holds:
\begin{itemize}
\itemsep-0.2em
    \item $(H, x)$ is a proper wheel of $G$ and $K = \{x\}$.
    
   \item $(H, x)$ is a short pyramid of $G$, $N(x) \cap H = \{x_1, x_2, y\}$ where $x_1x_2$ is an edge, and $K = \{x, y\}$.
    
    \item $(H, x, x_2)$ is a terminal twin wheel of $G$, $(H, x)$ is $x_2$-poor, and $K = \{x\}$.

\end{itemize}
If $F = (H, K)$ is a forcer, we say that $K$ is the \emph{center} of $F$.  
{  The forcer described in the first bullet is referred to as a {\em proper wheel forcer}, the one in the second bullet as a {\em short pyramid forcer},
and the one in the third bullet as a {\em twin wheel forcer}.}
A forcer $F = (H, K)$ is \emph{strong} if it is not a twin wheel {   forcer.} The following lemma shows that forcers generate {  clique} star cutsets. 

\begin{lemma}
Let $G$ be a $C_4$-free odd-signable graph and let $F = (H, K)$ be a forcer in $G$. Then, $K$ is the center of a clique star cutset in $G$. 
\label{lemma:forcer_starcutset}
\end{lemma}
\begin{proof}
{  If $(H,x)$  is a proper wheel that is not a universal wheel, then by Lemma \ref{lemma:proper_wheel_forcer}, $x$ together with some of its neighbors is
a clique star cutset in $G$. If $(H,x)$ is a universal wheel, then by Lemma  \ref{lemma:univ_wheel_forcer}, $x$ together with some of its neighbors is
a clique star cutset in $G$. If $(H,x)$ is a short pyramid and $y$ is the common node of the two long sectors of $(H,x)$, then by Lemma \ref{lemma:short_pyramid_forcer}, $x,y$ and its neighbors form a clique star cutset in $G$. It follows that if
$F = (H, K)$ is a strong forcer, then the result holds.} 
Therefore, assume $F = (H, K)$ is a twin wheel forcer. It follows that there exist $x \in V(G), x_2 \in V(H)$ such that $(H, x, x_2)$ is a terminal twin wheel, $(H, x)$ is $x_2$-poor, and $K = \{x\}$. Then, it follows that $N[K] \setminus x_2$ is a {  clique} star cutset that separates $x_2$ from $H \setminus N[K]$.
\end{proof}

The following lemma shows that if $F = (H, K)$ is a forcer and $S_K = (A_K, C_K, B_K)$ is the canonical star separation for $K$, then $A_K \cap H \neq \emptyset$. 
\begin{lemma}
\label{lemma:forcer_intersects_A}
\label{lemma:twin_wheel_AKcapH}
Let $G$ be a $C_4$-free odd-signable graph. Let $F = (H, K)$ be a forcer in $G$ and let $S_K = (A_K, C_K, B_K)$ be a canonical star separation for $K$. Then, $A_K \cap H \neq \emptyset$. Furthermore, if for $c \in [\frac{1}{2}, 1)$, $G$ has no $(w, c, 2)$-balanced separator, then $S_K$ is a proper star separation.
\end{lemma}
\begin{proof}
Let $(H, x)$ be the wheel such that $F = (H, K)$. Suppose first that $(H,x)$ is a wheel such that there exist two long sectors $S_1,S_2$ of $(H,x)$. Lemmas \ref{lemma:proper_wheel_forcer} and \ref{lemma:short_pyramid_forcer} imply that $N[K]$ separates $S_1 \setminus N[K]$ from $S_2 \setminus N[K]$. It follows that for some $i \in \{1,2\}$, $S_i \cap A_K \neq \emptyset$, and so $H \cap A_K \neq \emptyset$.

Next, suppose that $(H, x)$ is a proper wheel with exactly one long sector $S$. If $B_K \cap H = \emptyset$, then $S^* \cap A_K \neq \emptyset$, so we may assume that $S^* \subseteq B_K$. By Lemma \ref{lemma:proper_wheel_forcer}, for some $a \in N(x) \cap H$, $a$ has no neighbor in $B_K$. Therefore, $a \in A_K$ and $A_K \cap H \neq \emptyset$. 

Now, suppose that $(H,x)$ is a universal wheel.   We may assume that $G \neq N[K]$ (since otherwise $B_K = \emptyset$ and $A_K = H)$. Then, it follows from Lemma \ref{lemma:univ_wheel_forcer} that for every component $C$ of $G \setminus N[K]$, there exists $a \in H$ such that $a$ has no neighbor in $C$. In particular, there exists $a \in H$ such that $a$ has no neighbor in $B_K$. Therefore, $a \not \in C_K$ and $a \not \in B_K$, so $a \in A_K$ and $H \cap A_K \neq \emptyset$.

Finally, suppose that $(H, x)$ is a twin wheel, and let $x_2$ be the clone of $x$ in $H$. Then, $(H, x, x_2)$ is a terminal twin wheel, $(H, x)$ is $x_2$-poor, and $K = \{x\}$. Consider $G \setminus N[K]$. If $(H \setminus \{x_1, x_2, x_3\}) \cap B_K = \emptyset$, then $A_K \cap H \neq \emptyset$, so assume $(H \setminus \{x_1, x_2, x_3\}) \subseteq B_K$. Since $(H, x)$ is $x_2$-poor, it follows that $x_2$ does not have a neighbor in $B_K$. Therefore, $x_2 \in A_K$, and $A_K \cap H \neq \emptyset$.

Now, suppose that $c \in [\frac{1}{2}, 1)$ and $G$ has no $(w, c, 2)$-balanced separator. Then, $G \setminus N[K] \neq \emptyset$, and thus $B_K \neq \emptyset$. Since $A_K \neq \emptyset$, it follows that $S_K$ is proper. 
\end{proof}

Let $G'$ be an induced subgraph of $G$. A forcer $F = (H, K)$ is \emph{active for $G'$} if $H \subseteq G'$ and $K \subseteq G'$.

\begin{lemma}
Let $\delta$ be a positive integer, $c \in [\frac{1}{2}, 1)$, and $m \in [0, 1]$, with $(1-c) + m(\delta + \delta^2) < \frac{1}{2}$. Let $G$ be a connected $C_4$-free odd-signable graph with maximum degree $\delta$,  {let $w:V(G) \to [0, 1]$ be a weight function on $G$ with $w(G) = 1$ and $w^{\max} \leq m$}, and suppose $G$ does not have a $(w, c, 2)$-balanced separator. Let $\mathcal{F}$ be a set of forcers, let $Y = \{K : (H, K) \in \mathcal{F}\}$ be the set of centers of $\mathcal{F}$, and let $\C$ be the collection of canonical star separations for centers in $Y$. Suppose $\C$ is laminar and let $(T_\C, \chi_\C)$ be the tree decomposition of $G$ corresponding to $\C$. Then, the central bag $\beta$ for $\C$ exists and no forcer in $\mathcal{F}$ is active for $\beta$.
\label{lemma:not_cutset_in_central_bag}
\end{lemma}
\begin{proof}
By Lemma \ref{lemma:skewed}, every separation in $\C$ is $(1-c)$-skewed. By Lemma \ref{lemma:central_bag}, the central bag $\beta$ for $\C$ exists ({  in particular,} $\beta$ is perpendicular to $\C$). Suppose $F = (H, K)$ is a forcer in $\mathcal{F}$ and let $S_K = (A_K, C_K, B_K)$ be the canonical star separation for $K$. Then, 
{  since $\beta$ is perpendicular to $\C$, $\beta \cap A_K=\emptyset$, and hence $\beta \subseteq C_K \cup B_K$.} By Lemma \ref{lemma:forcer_intersects_A}, it follows that $H \cap A_K \neq \emptyset$, so $H \not \subseteq \beta$ and $F$ is not active for $\beta$. 
\end{proof}

The following theorem generalizes the results of Lemma \ref{lemma:not_cutset_in_central_bag}. 
{  Recall the definition of clique-free bag from the end of Section \ref{sec:clique}: the clique-free bag of a graph $G$ is an induced subgraph $\alpha$ of $G$, formed by taking repeated central bags, such that $\alpha$ does not have a clique cutset. (See Theorem \ref{thm:clique_free_bag} for details).}

 \begin{theorem}
 \label{thm:bigthm}
 Let $\delta, d$ be positive integers, let $k$ be a nonnegative integer, let $f(2, \delta) = 2(\delta + 1)^2 + 1$, let $c \in [\frac{1}{2}, 1)$, and let $m \in [0, 1]$, with $d > 2f(2, \delta)\delta + 2\delta$, and $(1-c) + \left[m + f(2, \delta)\delta 2^{\delta}(1-c)\right](\delta + \delta^2) < \frac{1}{2}$. Let $G$ be a connected $C_4$-free odd-signable graph with maximum degree $\delta$,  {let $w:V(G) \to [0, 1]$ be a weight function on $G$ with $w(G) = 1$ and $w^{\max} \leq m$}, and suppose that $G$ does not have a $(w, c, d)$-balanced separator. Let $\mathcal{F}$ be a set of forcers of $G$. Then, there exists a sequence $(\beta_1, w_1), \hdots, (\beta_{2k+1}, w_{2k+1})$, where $\beta_{2k+1} \subseteq \beta_{2k} \subseteq \hdots \subseteq \beta_1 \subseteq \beta_0 = G$, $k \leq f(2, \delta)$, and for $i \in \{1, \hdots, 2k+1\}$,  $w_i$ is a weight function on $\beta_i$, {  with $w_i(\beta_i)=1$}, such that: 
 \begin{itemize}
 \itemsep -0.1em

     \item for $i \in \{0, \hdots, k\}$, $\beta_{2i+1}$ is the clique-free bag for $\beta_{2i}$, 
     
     \item for $i \in \{0, \hdots, k-1\}$, $\beta_{2i+2}$ is the central bag for a tree decomposition corresponding to a laminar collection of proper star separations of $\beta_{2i+1}$ with clique centers of size 1 or 2 (of size 1 if $\mathcal{F}$ does not contain a short pyramid forcer),
     
     \item {  for $i \in \{0, \hdots, k\}$}, $\beta_{2i+1}$ is connected and does not have a $(w_{2i+1}, c, d_{2i+1})$-balanced separator, for $d_{2i+1} = d - 2i\delta - 2(\delta - 1)$, and {  for $i \in \{0, \hdots, k-1\}$}, $\beta_{2i+2}$ is connected and does not have a $(w_{2i+2}, c, d_{2i+2})$-balanced separator for $d_{2i+2} = d - 2(i+1)\delta$,

     \item $w_{2k+1}^{\max} \leq w^{\max} + f(2, \delta)\delta 2^{\delta}(1-c) + (\delta - 1)2^{\delta}(1-c)$,
     
     \item no forcer in $\mathcal{F}$ is active for $\beta_{2k+1}$,
     
     \item $\beta_{2k+1}$ has no clique cutset.
 \end{itemize} 
 \end{theorem}
    
 \begin{proof}
Let $Y = \{K : (H, K) \in \mathcal{F}\}$ be the set of centers of forcers in $\mathcal{F}$. For all $K \in Y$, $|K| \in \{1, 2\}$, and if $(H, K)$ is not a short pyramid forcer, then $|K| = 1$. Let $(Y_1, \hdots, Y_{f(2, \delta)})$ be a partition of $Y$ as in Lemma \ref{lemma:csc_centers_partition} and let $\mathcal{F}_1, \hdots, \mathcal{F}_{f(2, \delta)}$ be a partition of $\mathcal{F}$ such that $Y_i = \{K : (H, K) \in \mathcal{F}_i\}$, {  for $i\in \{ 1, \ldots ,f(2,\delta)\}$.} 
Let $\beta_1$ be the clique-free bag for $G$ and let $w_1$ be the weight function on $\beta_1$ from Theorem \ref{thm:clique_free_bag}. By Theorem \ref{thm:clique_free_bag}, $\beta_1$ has no clique cutset and no $(w_1, c, d-2(\delta-1))$-balanced separator, where $w_1(\beta_1) = 1$ and $w_1^{\max} \leq w^{\max} + (\delta-1) 2^{\delta}(1-c)$. If no forcer in $\mathcal{F}$ is active for $\beta_1$, then $k=0$, and the sequence ends. 

Otherwise, assume that there is a forcer in $\mathcal{F}_1$ active for $\beta_1$. Let $X_1 = \{S_K : K \in Y_1\}$ be the set of canonical star separations of $\beta_1$ for centers in $Y_1$. Since $\beta_1$ has no $(w_1, c, d-2(\delta-1))$-balanced separator and $d - 2(\delta-1) \geq 2$, by Lemma \ref{lemma:skewed}, every clique $K$ appears as a center of at most one separation in $X_1$ and every separation in $X_1$ is $(1-c)$-skewed. Since $\beta_1$ has no clique cutset and cliques in $Y_1$ are pairwise anticomplete, and by Lemma \ref{lemma:forcer_intersects_A} the separations in $X_1$ are all proper, it follows from Lemma \ref{lemma:star_sepns_cross} that $X_1$ is laminar.
Since $X_1$ is a laminar collection of star separations of $\beta_1$ and $(1-c) + w_1^{\max}(\delta + \delta^2) \leq (1-c) + [w^{\max} + (\delta-1) 2^{\delta}(1-c)](\delta + \delta^2) < \frac{1}{2}$, by Lemma \ref{lemma:biglemma}, the central bag $\beta_2$ for $X_1$ exists and $\beta_2$ does not have a $(w_{X_1}, c, d - 2\delta)$-balanced separator.  Let $w_2 = w_{X_1}$ be the weight function on $\beta_2$ with respect to $T_{X_1}$, where $T_{X_1}$ is the tree decomposition of $\beta_1$ corresponding to $X_1$. By Lemma \ref{lemma:weight_max}, $w_2(\beta_2) = 1$ and $w_2^{\max} \leq w_1^{\max} + 2^{\delta}(1-c) \leq w^{\max} + \delta2^{\delta}(1-c)$. By Lemma \ref{lemma:not_cutset_in_central_bag}, it follows that no forcer in $\mathcal{F}_1$ is active for $\beta_2$. By Lemma \ref{lemma:central_bag}, $\beta_2$ is connected.

  For $i > 0$, we define $(\beta_{2i+1}, w_{2i+1})$ and $(\beta_{2i+2}, w_{2i+2})$ inductively. For $i \in \{1, \hdots, f(2, \delta)\}$, suppose $(\beta_{2i}, w_{2i})$ are such that $\beta_{2i}$ is connected and has no $(w_{2i}, c, d_{2i})$-balanced separator for $d_{2i} =d - 2i\delta \geq 1$, $w_{2i}(\beta_{2i}) = 1$, and $w_{2i}^{\max} \leq w^{\max} + i\delta2^\delta(1 - c)$. Further, suppose there exists $I_{i} \subseteq \{1, \hdots, f(2, \delta)\}$ such that $i \leq |I_{i}| < f(2, \delta)$, no forcer in $\bigcup_{j \in I_{i}} \mathcal{F}_j$ is active for $\beta_{2i}$, and for all $j \in \{1, \hdots, f(2, \delta)\} \setminus I_{i}$, there is a forcer in $\mathcal{F}_j$ active for $\beta_{2i}$. 

Since $d > 2f(2, \delta)\delta + 2\delta$ and $i < f(2, \delta)$, it follows that $d_{2i} = d - 2i\delta > 2\delta > 2\delta - 2$. Also, since $\beta_{2i}$ has no $(w_{2i}, c, d_{2i})$-balanced separator and $(1-c) + [w_{2i}^{\max} + \delta 2^{\delta} (1-c)](\delta + \delta^2) \leq (1-c) + [w^{\max} + f(2, \delta)\delta 2^{\delta}(1-c)](\delta + \delta^2) < \frac{1}{2}$, the conditions of Theorem \ref{thm:clique_free_bag} for $\beta_{2i}$ are satisfied. Let $\beta_{2i+1}$ be the clique-free bag for $\beta_{2i}$ and let $w_{2i+1}$ be the weight function on $\beta_{2i+1}$ from Theorem \ref{thm:clique_free_bag}. By Theorem \ref{thm:clique_free_bag}, $\beta_{2i+1}$ does not have a $(w_{2i+1}, c, d_{2i}-2(\delta-1))$-balanced separator, where $w_{2i+1}(\beta_{2i+1}) = 1$ and $w_{2i+1}^{\max} \leq w_{2i}^{\max} + (\delta-1) 2^{\delta}(1-c) \leq w^{\max} + i\delta2^{\delta}(1-c) + (\delta - 1)2^{\delta}(1-c)$. Let $d_{2i+1} = d_{2i} - 2(\delta-1)$. If no forcer in $\mathcal{F}$ is active for $\beta_{2i+1}$, then $k=i$, and the sequence ends. Otherwise, let $\sigma_i \in \{1, \hdots, f(2, \delta)\} \setminus I_{i}$ be such that there is a forcer in $\mathcal{F}_{\sigma_i}$ that is active for $\beta_{2i+1}$. Let $X_{\sigma_i} = \{S_K : K \in Y_{\sigma_i}\}$ be the set of canonical star separations of $\beta_{2i+1}$ for centers in $Y_{\sigma_i}$. Since $\beta_{2i+1}$ has no $(w_{2i+1}, c, d_{2i+1})$-balanced separator, by Lemma \ref{lemma:skewed}, every clique $K$ appears as the center of at most one separation in $X_{\sigma_i}$ and every separation in $X_{\sigma_i}$ is $(1 - c)$-skewed. Since $\beta_{2i+1}$ has no clique cutset and cliques in $Y_{\sigma_i}$ are pairwise anticomplete and by Lemma \ref{lemma:forcer_intersects_A} the separations in $Y_{\sigma_i}$ are all proper, it follows from Lemma \ref{lemma:star_sepns_cross} that $X_{\sigma_i}$ is laminar. Finally, $d_{2i+1} > 2$ and, since $i < f(2, \delta)$, $(1-c) + w_{2i+1}^{\max}(\delta + \delta^2) \leq (1-c) + \left[w^{\max} + f(2, \delta)\delta 2^{\delta}(1-c)\right](\delta + \delta^2) < \frac{1}{2},$ so by Lemma \ref{lemma:biglemma}, the central bag $\beta_{2i+2}$ for $X_{\sigma_i}$ exists and $\beta_{2i+2}$ does not have a $(w_{X_{\sigma_i}}, c, d_{2i+2})$-balanced separator, where $d_{2i+2} = d_{2i+1} - 2 = d - 2(i+1)\delta$. Let $w_{2i+2} = w_{X_{\sigma_i}}$ be the weight function on $\beta_{2i+2}$ with respect to $T_{X_{\sigma_i}}$, where $T_{X_{\sigma_i}}$ is the tree decomposition of $\beta_{2i+1}$ corresponding to $X_{\sigma_i}$. By Lemma \ref{lemma:weight_max}, $w_{2i+2}(\beta_{2i+2}) = 1$ and $w_{2i+2}^{\max} \leq w_{2i+1}^{\max} + 2^\delta(1-c) \leq w^{\max} + (i+1)\delta 2^\delta(1 - c)$. By Lemma \ref{lemma:central_bag}, $\beta_{2i+2}$ is connected. Let $I_{i+1}$ be the set of all $j \in \{1, \hdots, f(2, \delta)\}$ such that no forcer in $\mathcal{F}_j$ is active for $\beta_{2i+2}$. Since $\beta_{2i+2} \subseteq \beta_{2i}$ and no forcer in $\bigcup_{j \in I_{i}} \mathcal{F}_j$ is active for $\beta_{2i}$, it follows that no forcer in $\bigcup_{j \in I_{i}} \mathcal{F}_j$ is active for $\beta_{2i+2}$. Further, since $\beta_{2i+2}$ is the central bag for a tree decomposition corresponding to $X_{\sigma_i}$, it follows from Lemma \ref{lemma:not_cutset_in_central_bag} that no forcer in $\mathcal{F}_{\sigma_i}$ is active for $\beta_{2i+2}$. Therefore, $|I_{i+1}| \geq i+1$, and $(\beta_{2i+2}, w_{2i+2})$ satisfies the conditions of the induction. It follows that the sequence $(\beta_1, w_1), \hdots, (\beta_{2k+1}, w_{2k+1})$ is well-defined, $k \leq f(2, \delta)$, $\beta_{2k+1}$ does not have a clique cutset, and no forcer in $\mathcal{F}$ is active for $\beta_{2k+1}$.
 \end{proof}

 We call $(\beta_1, w_1), \hdots, (\beta_{2k+1}, w_{2k+1})$ as in Theorem \ref{thm:bigthm} an \emph{$\mathcal{F}$-decomposition} of $G$, and $\beta_{2k+1}$ the \emph{terminal bag} for $(\beta_1, w_1), \hdots,$ $(\beta_{2k+1}, w_{2k+1})$. A graph $G$ is \emph{clean} if $G$ does not contain a strong forcer. The following theorem shows that if $\mathcal{F}$ is the collection of all strong forcers of $G$ and $\beta_{2k+1}$ is the terminal bag for a $\mathcal{F}$-decomposition, then $\beta_{2k+1}$ is clean. 

\begin{theorem}\label{thm:strong_forcers_clean}
Let $\delta, d$ be positive integers, let $f(2, \delta) = 2(\delta + 1)^2 + 1$, let $c \in [\frac{1}{2}, 1)$, and let $m \in [0, 1]$, with $d > 2f(2, \delta)\delta + 2\delta$, and $(1-c) + \left[m + f(2, \delta)\delta 2^{\delta}(1-c)\right](\delta + \delta^2) < \frac{1}{2}$. Let $G$ be a connected $C_4$-free odd-signable graph with maximum degree $\delta$, let $w:V(G) \to [0, 1]$ be a weight function on $G$ with $w(G) = 1$ and $w^{\max} \leq m $, and suppose $G$ does not have a $(w, c, d)$-balanced separator. Let $\mathcal{F}$ be the set of all strong forcers of $G$, and let $(\beta_1, w_1), \hdots, (\beta_{2k+1}, w_{2k+1})$ be an $\mathcal{F}$-decomposition. Then, the terminal bag $\beta_{2k+1}$ is clean.
\end{theorem}
\begin{proof}
Suppose $\beta_{2k+1}$ contains a strong forcer $F = (H, K)$. Then, $F$ is a strong forcer in $G$, so $F \in \mathcal{F}$. By Theorem \ref{thm:bigthm}, it follows that $F$ is not active for $\beta_{2k+1}$, a contradiction.
\end{proof}

\section{Twin wheels in clean graphs}
\label{sec:twin_wheel_clean}

 {In this section we study twin wheels. It turns out that not all twin wheels are
clique star cutset forcers, but some of them (``terminal'' ones) are. 
The goal of this section is to show that the central bag for the 
collection of all twin wheel forcers of a clean graph $G$
does not contain a terminal twin wheel.}

Let $G$ be a clean $C_4$-free odd-signable graph.
The following two lemmas describe the behavior of twin wheels in $G$. Lemma \ref{lemma:u_x_poor} follows from the proof of Lemma 8.4 in \cite{daSilva2013Decomposition2-joins} and Lemma \ref{lemma:paths_shapes} follows from the proof of Theorem 1.5 in \cite{daSilva2013Decomposition2-joins}.
 {  For completeness we include their proofs.}

\begin{lemma}
{\em (\cite{daSilva2013Decomposition2-joins})}
\label{lemma:u_x_poor}
Let $G$ be a clean $C_4$-free odd-signable graph. Let $(H, x)$ be a twin wheel contained in $G$. Let $x_1\dd x_2 \dd x_3$ be the subpath of $H$ such that $N(x) \cap H = \{x_1, x_2, x_3\}$. Suppose there exists a vertex $u \in V(G)$ such that $N(u) \cap (H \cup x) = \{x, x_1, x_1'\}$, where $x_1'$ is the neighbor of $x_1$ in $H \setminus x_2$. Then, $(H, x)$ is $x_2$-poor. 
\end{lemma}

{ 
\begin{proof}
Let $x_1 \dd p_1 \dd \hdots \dd p_k \dd x_3$ be the long sector of $(H,x)$, and let $P=p_1 \dd \hdots \dd p_k$. Suppose that $(H,x)$ is $x_2$-rich.
Then there exists a path $Q=q_1\dd \hdots \dd q_l$ in $G\setminus (N[x]\setminus \{ x_2\})$ from $x_2$ to $P$.
We may assume that $Q$ is chosen to be the minimal such path. Then, $q_l$ has a neighbor in $P$, $x_1$ and $x_3$ are the only nodes of $H$ that may have a neighbor 
in $Q\setminus q_l$, $x_2$ is adjacent to $q_1$, and $x_2$ does not have a neighbor in $Q\setminus q_1$.
Let $p_i$ (resp. $p_{i'}$) be the neighbor of $q_l$ in $P$ with lowest (resp. highest) index. 
\\
\\
{\em (1) Both $u$ and $x_1$ have a neighbor in $Q$.}

$N(u)\cap Q\neq \emptyset$, else $Q\cup \{ p_1, \ldots ,p_i,x_1,x_2,u,x\}$ induces a proper wheel with center $x_1$, contradicting the assumption that $G$ is clean.
Now suppose that $N(x_1)\cap Q=\emptyset$. Let $H'$ be the hole induced by $Q\cup \{ p_1,\ldots ,p_i,x_1,x_2\}$.
Since $G$ is clean, $(H',u)$ is a twin wheel, and hence $i=1$ and $N(u)\cap Q=\{ q_l\}$.
Since $\{ u,x,x_3,q_l\}$ cannot induce a $C_4$, $x_3q_l$ is not an edge.
 Since $\{ u,x,x_2,q_1\}$ cannot induce a $C_4$, $l>1$.
 Suppose $i'=1$. If $N(x_3)\cap Q=\emptyset$, then $Q\cup H$ induces a theta. So $N(x_3)\cap Q\neq \emptyset$. 
 Let $q_s$ be the node of $N(x_3)\cap Q$ with highest index.
 Then $\{ q_s, \ldots q_l,p_1,x_1,x,x_3,u\}$ induces a proper wheel with center $u$, a contradiction.
 So $i'>1$. But then $\{ q_l,p_{i'},\ldots , p_k, u,x_1,x_2,x_3,x\}$ induces a proper wheel with center $x$, a contradiction.
 This proves (1).
 \\
 \\
 {\em (2) $N(x_3)\cap Q=\emptyset$.}
 
 Suppose $x_3$ has a neighbor in $Q$. By (1), let $q_s$ (resp. $q_t$) be the node of $Q$ with the lowest index adjacent to $x_1$ (resp. $u$).
 If $s\leq t$, then $\{ q_1, \ldots ,q_t,u,x,x_1,x_2\}$ induces a proper wheel with center $x_1$.
 So $s>t$. In particular, $t<l$ and $s>1$.
 If $x_3$ has a neighbor in $Q\setminus q_l$, then $(Q\setminus q_l) \cup P \cup \{ u,x,x_3\}$ contains a theta between $u$ and $x_3$.
 So $x_3$ has no neighbor in $Q\setminus q_l$, and hence $N(x_3)\cap Q=\{ q_l\}$.
 Let $H'$ be the hole induced by $Q\cup \{ x_2,x_3\}$. Since $H'\cup x_1$ cannot induce a theta, $(H',x_1)$ is a wheel.
 Since $s>1$, $(H',x_1)$ is a proper wheel or a short pyramid, contradicting that $G$ is clean.
 This proves (2).
 
 \vspace{2ex}
 
 By (1), let $q_s$ (resp. $q_t$) be the node of $Q$ with lowest index adjacent to $x_1$ (resp. $u$).
 If $s=1$ then $\{ q_1, \ldots ,q_t,x,x_2,x_1,u\}$ induces a proper wheel with center $x_1$, a contradiction.
 So $s>1$. By (2), $Q\cup \{ p_{i'},\ldots ,p_k,x_2,x_3\}$ induces a hole $H'$. But then, since $s>1$, either $H'\cup x_1$ induces a theta,
 or $(H',x_1)$ is  a proper wheel or a short pyramid, a contradiction.

\end{proof}

\begin{lemma}
{\em (\cite{daSilva2013Decomposition2-joins})}
\label{lemma:paths_shapes}
Let $G$ be a clean $C_4$-free odd-signable graph. Let $(H, x)$ be a twin wheel contained in $G$, let $N(x) \cap H = \{x_1, x_2, x_3\}$, where $x_2$ is the clone of $x$ in $H$, and suppose $(H, x, x_2)$ is not a terminal twin wheel. Then, there exists a path $P = p_1 \dd \hdots \dd p_k$ in $G \setminus (H \cup x)$ such that $N(p_1) \cap (H \cup x) = \{x\}$, $N(p_k) \cap (H \cup x)$ is an edge of $H \setminus \{x_1, x_2, x_3\}$, and $P^*$ is anticomplete to $H \cup x$. Similarly, there exists a path $Q = q_1 \dd \hdots \dd q_j$ in $G \setminus (H \cup x)$ such that $N(q_1) \cap (H \cup x) = \{x_2\}$, $N(q_j) \cap (H \cup x)$ is an edge of $H \setminus \{x_1, x_2, x_3\}$, and $Q^*$ is anticomplete to $H \cup x$. 
\end{lemma}
\begin{proof} 
  Since $(H,x)$ is not terminal, it follows that  $(H,x)$ is $x$-rich and
  $x_2$-rich. 
  Let $x_1\dd q_1 \dd \hdots \dd q_l\dd x_3$ be the long sector of $(H,x)$,
  and let $Q$ be the path $q_1 \dd \dots \dd q_l$.
Then by Lemma \ref{lemma:u_x_poor}, there does not exist a node $u$ such that $N(u)\cap (H\cup x)=\{ x,x_1,q_1\}$, and by symmetry,
there does not exist a node $u$ such that $N(u)\cap (H\cup x)=\{ x,x_3,q_l\}$.
Since $(H,x)$ is $x$-rich, there exists a path $P=p_1\dd \hdots \dd p_k$ in $G\setminus (N[x_2]\setminus \{ x\} )$ from $x$ to $Q$.
We may assume that $P$ is chosen to be the minimal such path.
Then, $p_k$ has a neighbor in $Q$, $x_1$ and $x_3$ are the only nodes of $H$ that may have a neighbor in $P\setminus p_k$, and $N(p_1)\cap (H\cup x)=\{ x\}$.
Let $q_i$ (resp. $q_{i'}$) be the neighbor of  $p_k$ in $Q$ with lowest (resp. highest) index.
\\
\\
{\em (1) $\{x_1,x_3\}$ is anticomplete to $P$.}

Suppose that one of $x_1,x_3$  has  a neighbor in $P$. Since
$\{x_2,x_2,x_3,p_k\}$ does not induce a $C_4$, 
not both $x_1,x_3$ are adjacent to  $p_k$. Since $H\cup (P\setminus p_k)$ does
not contain a theta between $x_1$ and $x_3$, it follows that at least one of 
$x_1,x_3$ is anticomplete to $P \setminus p_k$. It follows (exchanging the
roles of $x_1,x_3$ if necessary) that we
may assume that $x_3$ has a neighbor in $P$, and $x_1$ is anticomplete
to $P \setminus p_k$.

Since $\{x_1,p_1,x_3,x_2\}$ does
not induce a $C_4$, it follows that if $k=1$, then $x_1$ is non-adjacent
to $p_k$. Consequently,   $P\cup \{ x_1,x,q_1,\ldots ,q_i\}$ induces a hole $H'$.
Since $H'\cup x_3$ does not induce a theta or a strong forcer,
$x_3$ is adjacent to $p_1$ and
$N(x_3) \cap H' \subseteq N(p_1) \cap H'$.
If $N(x_3)\cap H'=\{ x, p_1\}$, then $H'\cup \{ x_2,x_3\}$ induces a proper wheel with center $x$.
So $N(x_3)\cap H'=N(p_1) \cap H'$. 

Let $H''$ be the hole induced by $(H'\setminus \{x,p_1\})\cup \{ x_2,x_3\}$. Then $(H'',x)$ is a twin wheel, and $N(p_1)\cap (H''\cup x)=\{ x,x_3,x_3'\}$,
where $x_3'$ is the neighbor of $x_3$ in $H'' \setminus x_2$.
Since $(H,x)$ is $x_2$-rich, there is a path $R$ in $G\setminus (N[x]\setminus x_2)$ from $x_2$ to $Q$.  It follows $R\cup \{ q_{i'}, \ldots ,q_l\}$ contains a path showing that $(H'',x)$ is  $x_2$-rich. But Lemma \ref{lemma:u_x_poor} (with $p_1$ playing the role of $u$) implies that  $(H'',x)$ is $x_2$-poor, a
contradiction. This proves (1).

\vspace{2ex}

If $k=1$ then (since by (1) $\{ x_1,x_3\}$ is anticomplete to $P$) $(H\setminus x_2)\cup P\cup x$ induces a theta or a strong forcer.
So $k>1$. If $i=i'$ or $p_ip_{i'}$ is not an edge, then the graph induced by $(H\setminus x_2)\cup P\cup x$ contains a theta between $x$ and either $p_k$ (when $i\neq i'$)
or $p_i$ (when $i=i'$). So  $p_ip_{i'}$ is an edge.
By symmetry between $x$ and $x_2$, the result follows..

\end{proof}

We now use \ref{lemma:paths_shapes} to show that twin wheel forcers can be used
in a way similar to strong forcers.}
\begin{theorem}
Let $\delta, d$ be positive integers, let $f(2, \delta) = 2(\delta + 1)^2 + 1$, let $c \in [\frac{1}{2}, 1)$, and let $m \in [0, 1]$, with $d > 2f(2, \delta)\delta + 2\delta$ and $(1-c) + [m + f(2, \delta)\delta 2^{\delta}(1-c)](\delta + \delta^2) < \frac{1}{2}$. Let $G$ be a connected clean $C_4$-free odd-signable graph with maximum degree $\delta$,  {let $w:V(G) \to [0, 1]$ be a weight function on $G$ with $w(G) = 1$ and $w^{\max} \leq m$}, and suppose $G$ does not have a $(w, c, d)$-balanced separator. Let $\mathcal{T}$ be the set of all twin wheel forcers in $G$ and let $(\beta_1, w_1), \hdots, (\beta_{2k+1}, w_{2k+1})$ be a $\mathcal{T}$-decomposition of $G$. Then, $\beta_{2k+1}$ does not contain a terminal twin wheel.
\label{thm:no_twin_wheel}
\end{theorem}
\begin{proof}
Let $\beta_0 = G$. \\

\noindent \emph{(1) For $i \in \{1, \hdots, 2k+1\}$, if $(H, x, x_2)$ is a terminal twin wheel in $\beta_{i}$, then $(H, x, x_2)$ is a terminal twin wheel in $\beta_{i-1}$.} 

 Let $(H, x, x_2)$ be a terminal wheel in $\beta_{i}$, with $N(x) \cap H = \{x_1, x_2, x_3\}$, and suppose $(H, x, x_2)$ is not a terminal wheel in $\beta_{i-1}$. Since $(H, x, x_2)$ is not a terminal twin wheel in $\beta_{i-1}$, by Lemma \ref{lemma:paths_shapes} there exists a path $P = p_1 \dd \hdots \dd p_m$ in $\beta_{i-1}$ such that $N(p_1) \cap (H \cup x) = \{x_2\}$, $N(p_m) \cap (H \cup x)$ is an edge of $H \setminus \{x_1, x_2, x_3\}$, and $P^*$ is anticomplete to $H \cup x$. Similarly, there exists a path 
 $Q = q_1 \dd \hdots \dd {  q_t}$ in $\beta_{i-1}$ such that $N(q_1) \cap (H \cup x) = \{x\}$, $N({  q_t}) \cap (H \cup x)$ is an edge of $H \setminus \{x_1, x_2, x_3\}$, and $Q^*$ is anticomplete to $H \cup x$. Since $(H, x, x_2)$ is a terminal twin wheel in $\beta_i$, we may assume that $V(P) \not \subseteq V(\beta_i)$. 
 {  If $i$ is odd, then by the definition of $\mathcal{T}$-decomposition, $\beta_i$ is the clique-free bag of $\beta_{i-1}$. By the definition of the clique-free bag,
 it follows that $\beta_i$ is an induced subgraph of $\beta_{i-1}$ obtained by decomposing $\beta_{i-1}$ with clique cutsets. Since
 $H \cup x \cup P$ does not have a clique cutset, it follows that $H \cup x \cup P$ is contained in $\beta_i$, a contradiction. Therefore,
 $i$ is even, and so by the definition of $\mathcal{T}$-decomposition,} $\beta_i$ is the central bag for a tree decomposition corresponding to a laminar collection of proper star separations in $\beta_{i-1}$. Let $p_0 = x_2$ and let $p_{m+1}$ be a neighbor of $p_m$ in $H$. Let $\ell \in \{1, \hdots, m\}$ and $j \in \{1, \hdots, m+1\}$ be such that $\ell < j$, $p_{\ell-1}, p_j \in \beta_i$, and ${  p_s} \not \in \beta_i$ for $\ell \leq {  s} < j$. It follows that $p_{\ell-1}$ and $p_j$ have neighbors in a connected component of $\beta_{i-1} \setminus \beta_i$. Since $\beta_i$ is the central bag for a tree decomposition corresponding to a collection of star separations in $\beta_{i-1}$, it follows that $p_{\ell-1}$ and $p_j$ are in a star cutset of $\beta_{i-1}$. In particular, there exists $v \in \beta_i$ such that $p_{\ell-1}, p_j \in N[v]$.  Since $P^*$ is anticomplete to $H \cup x$, it follows that $v \not \in H$. 

Since there does not exist a path from $x_2$ to $H \setminus \{x_1, x_2, x_3\}$ in $\beta_i$ not containing a neighbor of $x$, it follows that $v$ is adjacent to $x$, and thus $p_{\ell-1}, p_j \neq v$. Let $N(p_m) \cap (H \cup x) = \{h_1, h_2\}$, where $h_1$ is on the path from $x_1$ to $h_2$ through $H \setminus x_2$. We may assume that if $v$ is adjacent to one of $h_1, h_2$, then $v$ is adjacent to $h_1$ and $h_1 = p_{m+1}$. Let $R$ be the path from $h_1$ to $x_1$ not containing $h_2$ in $H$. Consider the hole $H'$ given by $x_1 \dd x_2 \dd p_1 \dd P \dd p_m \dd h_1 \dd R \dd x_1$.  Then, $v$ has two non-adjacent neighbors $p_{\ell-1}$ and $p_j$ in $H'$. Since $G$ is clean and theta-free, it follows that $(H', v)$ is a twin wheel. 
{  Since $v$ is adjacent to both $p_{\ell -1}$ and $p_j$, and $p_{\ell -1}p_j$ is not an edge, and $(H',v)$ is a twin wheel, either all the neighbors of $v$ in $H'$
are contained in $R\cup x_2$, or they are all contained in $P\cup \{ p_0,p_{m+1}\}$. Since $v$ has at least 2 neighbors in $P\cup \{ p_0,p_{m+1}\}$, it follows that}
either $p_j = h_1 = p_{m+1}$, $p_{\ell-1} = p_0$, and $N(v) \cap (H \cup P) = \{x_1, x_2, h_1\}$, where $h_1x_1$ is an edge and $v$ has no other neighbors in $H$ because $G$ is clean; or $j = \ell+1$ and $N(v) \cap H' = \{p_{\ell-1}, p_{\ell}, p_{\ell+1}\}$. In the first case, $h_2 \in H \setminus N[v]$ and $p_mh_2$ is an edge, so $P$ and $H \setminus N[v]$ are in the same connected component of $\beta_{i-1} \setminus N[v]$. 
{  Since $H\subseteq \beta_i$, it follows that}
$P \subseteq \beta_i$, a contradiction. Therefore, the second case holds. Now, consider the hole $H''$ given by $x_1 \dd x_2 \dd p_1 \dd P \dd p_{\ell-1} \dd v \dd p_j \dd P \dd p_m \dd h_1 \dd R \dd x_1$. Then, $N(x) \cap H'' = \{x_1, x_2, v\}$, and since $G$ is clean, $(H'', x)$ is not a short pyramid. Therefore, $p_{\ell-1} = x_2 = p_0$. 

Let $S$ be the path from $h_2$ to $x_3$ in $H \setminus \{h_1\}$. Since $N(v) \cap H' = \{p_0, p_1, p_2\}$, it follows that $v$ has no neighbors in $P \setminus \{p_1, p_2\}$. Further, since $v$ has three neighbors $x_2, p_1, p_2$ in the hole given by $x_2 \dd x_3 \dd S \dd h_2 \dd p_m \dd P \dd p_1 \dd x_2$, it follows that $v$ has no neighbors in $S$. Therefore, let $H'''$ be the hole given by $x \dd v \dd p_2 \dd P \dd p_m \dd h_2 \dd S \dd x_3 \dd x$. Then, $(H''', x_2)$ is a twin wheel, where $x$ is the clone of $x_2$ in $H'''$. Furthermore, there is a path contained in $Q \cup (P \setminus p_1) \cup (H \setminus x_2)$ from $x$ to $H''' \setminus \{v, x, x_3\}$ containing no neighbor of $x_2$ other than $x$, so $(H''', x_2)$ is $x$-rich. But $N(p_1) \cap {  (H'''\cup x_2)} = \{p_2, v, x_2\}$, contradicting Lemma \ref{lemma:u_x_poor}. This proves (1). \\

Suppose that $\beta_{2k+1}$ contains a terminal twin wheel $(H, x, x_2)$. By (1), it follows that $(H, x, x_2)$ is a terminal twin wheel in $G$, so we may assume that $F = (H, \{x\})$ is a twin wheel forcer in $G$. Then, by Theorem \ref{thm:bigthm}, $F$ is not active for $\beta_{2k+1}$, a contradiction. Therefore, $\beta_{2k+1}$ does not contain a terminal twin wheel. 
\end{proof}

The following lemma shows that if $G$ is a graph with no balanced separator, no clique cutset, and no forcer, then $G$ has no star cutset. 

\begin{lemma}
Let $c \in [\frac{1}{2}, 1)$. Let $G$ be a theta-free graph,  {let $w:V(G) \to [0, 1]$ be a weight function on $G$ with $w(G) = 1$ and $w^{\max} \leq m$}, and suppose that $G$ has no $(w, c, 1)$-balanced separator, $G$ has no clique cutset, and $G$ has no forcer. Then $G$ has no star cutset. 
\label{lemma:no_sc}
\end{lemma}
\begin{proof}
Suppose $G$ has a star cutset $C'$ centered at $v$ and let $(A', C', B')$ be a star separation such that $A', B' \neq \emptyset$. Let $(A, C, B)$ be the canonical star separation for $\{v\}$. Since $G$ has no $(w, c, 1)$-balanced separator, $G \setminus N[v] \neq \emptyset$, and therefore $B \neq \emptyset$. Without loss of generality let $B \subseteq B'$. Then, $A' \subseteq A$, and therefore $A \neq \emptyset$. 

Let $A^*$ be a component of $A$. Since $G$ does not have a clique cutset, it follows that there exist $u_1, u_2 \in N(A^*)$ such that $u_1u_2 \not \in E(G)$. Let $P$ be a path from $u_1$ to $u_2$ through $B$ and let $Q$ be a shortest path from $u_1$ to $u_2$ through $A^*$. Let $H$ be the hole given by $u_1 \dd Q \dd u_2 \dd P \dd u_1$. Then, $v$ has two non-adjacent neighbors in $H$. Because $G$ is clean and theta-free, it follows that $(H, v)$ is not a proper wheel or a short pyramid. Therefore, $(H, v)$ is a twin wheel, 
{  and since by definition of canonical star separation $v$ has no neighbor in $B$,}
$Q = u_1 \dd a \dd u_2$ for some vertex $a \in A^*$, and $a$ is the clone of $v$ in $H$. Since every path from $a$ to $B$ intersects $N[v]$, it follows that $(H, v)$ is $a$-poor, so $(H, v, a)$ is a terminal twin wheel in $G$, a contradiction. 
\end{proof}

\section{Graphs with no star cutset} 
\label{sec:no_sc}

In this section, we show that if $G$ is a $C_4$-free odd-signable graph with bounded degree and no star cutset, then $G$ has {  bounded treewidth}. A partition $(X_1,X_2)$ of the vertex set of a graph $G$ is a {\em 2-join} if for $i=1,2$
there exist disjoint nonempty $A_i,B_i\subseteq X_i$ satisfying the following:
\begin{itemize}
\itemsep -0.2em
\item $A_1$ is complete to $A_2$, $B_1$ is complete to $B_2$, and there are no other edges 
between $X_1$ and $X_2$;
\item for $i=1,2$, $G[X_i]$ contains a path with one end in $A_i$, one end in $B_i$ and interior in $X_i\setminus (A_i\cup B_i)$
and $G[X_i]$ is not a path.
\end{itemize}

We say that $(X_1,X_2,A_1,B_1,A_2,B_2)$ is a {\em split} of the 2-join $(X_1,X_2)$. A {\em long pyramid} is a pyramid all of whose three paths are of length at least 2. An {\em extended nontrivial basic} graph $R$ is defined as follows:
\begin{itemize}
\itemsep -0.2em
\item $V(R)=V(L)\cup \{ x,y\}$.
\item $L$ is the line graph of a tree $T$.
\item $x$ and $y$ are adjacent, and $\{ x,y\}\cap V(L)=\emptyset$.
\item $L$ contains at least two maximal cliques of size at least 3.
\item The vertices of $L$ corresponding to the edges incident with vertices of degree 1 in $T$
are called {\em leaf vertices}. Each leaf vertex of $L$ is adjacent to exactly one of $\{ x,y\}$
and no other vertex of $L$ is adjacent to a vertex of $\{ x,y\}$.
\item These are the only edges in $R$.
\end{itemize}

We observe that in order to prove the decomposition theorem for $C_4$-free odd-signable graphs, 
extended nontrivial basic graphs are defined in a more complicated way in \cite{daSilva2013Decomposition2-joins},
but for what we want to prove here the above definition suffices. Let ${\cal B}^*$ be the class of graphs that consists of cliques, holes, long pyramids and extended nontrivial basic graphs. 

\begin{theorem}\label{C4fos-decomp}
{\em (\cite{daSilva2013Decomposition2-joins})} A $C_4$-free odd-signable 
graph either belongs to ${\cal B}^*$
or it has a star cutset or a 2-join.
\end{theorem}
Let $G$ be a graph and $(X_1,X_2,A_1,B_1,A_2,B_2)$  a split of a 2-join of $G$. The {\em blocks of decomposition} of $G$
with respect to $(X_1,X_2)$ are graphs $G_1$ and $G_2$ defined as follows.
Block $G_1$ is obtained from $G[X_1]$ by adding a {\em marker path} $P_2=a_2\dd \ldots \dd b_2$ of length 3 such that 
$a_2$ is complete to $A_1$, $b_2$ is complete to $B_1$, and these are the only edges between $P_2$ and $X_1$.
 Block $G_2$ is obtained analogously from $G[X_2]$ by adding a marker path $P_1=a_1\dd \ldots \dd b_1$.

The following lemma follows from the proofs of Lemmas 3.5 and 3.7 in \cite{Trotignon2012Combinatorial2-joins}.

\begin{lemma}\label{l1}
{\em (\cite{Trotignon2012Combinatorial2-joins})}
Let $G$ be a $C_4$-free graph with no star cutset, let $(X_1,X_2)$ be a 2-join of $G$, and $G_1$ and $G_2$ the corresponding blocks of decomposition.
Then $G_1$ and $G_2$ do not have star cutsets.
\end{lemma}

Below, we prove that if $G$ is a $C_4$-free odd-signable graph and $(X_1, X_2, A_1, B_1, A_2, B_2)$ is a split of a 2-join of $G$, then the blocks of decomposition of $G$ with respect to $(X_1, X_2)$ are also $C_4$-free odd-signable.

\begin{lemma}\label{l2}
Let $G$ be a $C_4$-free odd-signable graph with no star cutset, let $(X_1,X_2)$ be a 2-join of $G$, and $G_1$ and $G_2$ the corresponding blocks of decomposition.
Then $G_1$ and $G_2$ are $C_4$-free odd-signable.
\end{lemma}
\begin{proof}
By constructions of the blocks, clearly $G_1$ and $G_2$ are $C_4$-free. So by Theorem \ref{os} it suffices to show that if $G_1$ contains an even wheel, theta
or a prism $\Sigma$, then $G$ contains an even wheel, theta or a prism. 
Let $(X_1,X_2,A_1,B_1,A_2,B_2)$ be the split of $(X_1,X_2)$, and let $P_2=a_2\dd \ldots \dd b_2$ be the marker path of $G_1$.
We may assume that $\Sigma \cap P_2\neq \emptyset$, since otherwise we are done. Suppose that $A_2$ is complete to $B_2$. By definition of 2-join, either $X_2 \setminus (A_2 \cup B_2) \neq \emptyset$, or, without loss of generality, $|B_2| \geq 2$. So for $u \in B_2$, $S = A_2 \cup B_1 \cup \{u\}$ is a star cutset {  in $G$} separating $X_1 \setminus B_1$ from $X_2 \setminus (A_2 \cup \{u\})$. Therefore, $A_2$ is not complete to $B_2$, so let $a\in A_2$ and $b\in B_2$ be such that $ab$ is not an edge.
By definition of 2-join, there exists a path $Q_2$ in $G[X_2]$ whose one end is in $A_2$, the other in $B_2$ and whose interior is in $X_2\setminus (A_2\cup B_2)$.

First suppose that $\Sigma=(H,x)$ is an even wheel.
If $H\subseteq X_1$ then without loss of generality $x=a_2$, and hence $(H,a)$ is an even wheel in $G$. So we may assume that $H\cap P_2\neq \emptyset$.
It follows that without loss of generality, {  $H\cap P_2\in \{\{ a_2\}$, $\{ a_2,b_2\}, P_2\}$}. It follows that $x\in X_1$. If $H\cap P_2=\{ a_2\}$ then let $H'=(H\setminus \{ a_2\} )\cup \{ a\}$;
if $H\cap P_2=\{ a_2,b_2\}$ then let $H'=(H\setminus \{ a_2,b_2\} )\cup \{ a,b\}$; and
if $H\cap P_2=P_2$ then let $H'=(H\setminus P_2)\cup Q_2$. Then clearly $(H',x)$ is an even wheel in $G$.

Now assume that $\Sigma$ is a theta  or a prism. Let $R_1,R_2,R_3$ be
the three paths of $\Sigma$. Note that any two of the paths induce a hole, and assume up to symmetry that out of the three holes so created, the hole $H=R_1\cup R_2$ 
has the largest intersection with $P_2$. Then without loss of generality $H\cap P_2=\{ a_2\}$, $\{ a_2,b_2\}$ or $P_2$. If $H\cap P_2=\{ a_2\}$ then let $H'=(H\setminus \{ a_2\} )\cup \{ a\}$;
if $H\cap P_2=\{ a_2,b_2\}$ then let $H'=(H\setminus \{ a_2,b_2\} )\cup \{ a,b\}$; and
if $H\cap P_2=P_2$ then let $H'=(H\setminus P_2)\cup Q_2$. Then clearly $H'$ is a hole in $G$.
By the choice of $H$ it follows that $|R_3\cap P_2|\leq 1$ and hence either $R_3\subseteq X_1$, or $H\cap P_2=\{ a_2\}$ and $R_3\cap P_2=\{ b_2\}$. 
In the first case clearly $H'\cup R_3$ is a theta or a prism, so assume that $H\cap P_2=\{ a_2\}$ and $R_3\cap P_2=\{ b_2\}$. Then, up to symmetry, $a_2\in R_2$. But then it follows that the hole $R_2\cup R_3$ has a larger intersection with $P_2$ than $H$,
a contradiction.
\end{proof}

Let $G$ be a graph.
A {\em flat path} in  $G$ is a path of $G$ of length at least 2 whose interior vertices all have degree 2 in $G$
and whose ends do not have a common neighbor outside this path.
A  {\em leaf} in a graph is a vertex of degree at most 1.
Let ${\cal D}$ be a class of graphs and ${\cal B}\subseteq {\cal D}$. Given a graph $G\in {\cal D}$, a rooted tree $T_G$ is a 
{\em 2-join decomposition tree for $G$ with respect to ${\cal B}$} if the following hold:
\begin{itemize}
\item Each vertex of $T_G$ is a pair $(H,{\cal M})$ where $H$ is a graph in ${\cal D}$ and ${\cal M}$ is a set of vertex-disjoint flat paths of $H$.
\item The root of $T_G$ is $(G,\emptyset)$.
\item Each non-leaf vertex of $T_G$ is $(G',{\cal M}')$ where $G'$ has a 2-join $(X_1,X_2)$ such that the edges between $X_1$ and $X_2$
do not belong to any flat path in ${\cal M}'$. Let ${\cal M}_1$ (respectively ${\cal M}_2$) be the set of all flat paths of ${\cal M}'$ that belong to $G[X_1]$ (respectively $G[X_2]$).
Let $G_1$ and $G_2$ be the blocks of decomposition of $G'$ with respect to 2-join $(X_1,X_2)$ with marker paths $P_2$ and $P_1$ respectively.
The vertex $(G',{\cal M}')$ has two children, which are $(G_1,{\cal M}_1\cup \{ P_2\})$ and $(G_2,{\cal M}_2\cup \{ P_1\})$.
\item Each leaf vertex of $T_G$ is $(G',{\cal M}')$ where $G'\in {\cal B}$. 
\end{itemize}

The following theorem follows from Lemma 4.6 in \cite{Trotignon2012Combinatorial2-joins}.

\begin{theorem}\label{t:tv}
{\em (\cite{Trotignon2012Combinatorial2-joins})}
Let $G$ be a graph and let ${\cal M}$ be a set of vertex-disjoint flat paths of $G$.
Then one of the following holds:
\begin{itemize}
\item[(i)] $G$ has no 2-join.
\item[(ii)] There exists a 2-join $(X_1,X_2)$ of $G$ such that for every path $P\in {\cal M}$, $P\subseteq X_1$ or $P\subseteq X_2$.
\item[(iii)] $G$ or a block of decomposition with respect to some 2-join of $G$ has a star cutset.
\end{itemize}
\end{theorem}

The following lemma shows that $C_4$-free odd-signable graphs with no star cutset have 2-join decomposition trees with respect to ${\cal B}^*$.

\begin{lemma}\label{l:2jdt}
If $G$ is a $C_4$-free odd-signable graph with no star cutset then $G$ has a 2-join decomposition tree with respect to ${\cal B}^*$.
\end{lemma}
\begin{proof}
If $G$ is a $C_4$-free odd-signable graph that has no star cutset then, by Lemmas \ref{l1} and \ref{l2}, blocks of decomposition of $G$ with respect to every 2-join
are $C_4$-free odd-signable and have no star cutset. So by repeated application of Theorem \ref{t:tv} there is a 2-join decomposition tree for $G$
in which the leaves correspond to $C_4$-free odd-signable graphs that have no star cutset and no 2-join, and hence by Theorem \ref{C4fos-decomp}
are graphs from ${\cal B}^*$, i.e. the result holds.
\end{proof}

The {\em rankwidth} of a graph $G$, denoted by $\rw(G)$, is a property of $G$ similar to treewidth. The definition of rankwidth can be found in  {\cite{rankwidth} (where it was first defined)}. The following theorem bounds the rankwidth of graphs that have a 2-join decomposition tree with respect to ${\cal B}^*$.

\begin{theorem}\label{t:le}
{\em (\cite{Le2017DetectingClasses, Le2018ColoringCutset})}
If ${\cal D}$ is a class of graphs such that every $G\in {\cal D}$ has a 2-join decomposition tree with respect to ${\cal B}^*$,
then $rw(G)\leq 3$.
\end{theorem}

\begin{corollary}\label{c:rw}
If $G$ is a $C_4$-free odd-signable graph with no star cutset then $rw(G)\leq 3$.
\end{corollary}
\begin{proof}
Follows from Theorem \ref{t:le} and Lemma \ref{l:2jdt}.
\end{proof}

The following theorem bounds the treewidth of $G$ by a function of the rankwidth of $G$ for graphs $G$ with no subgraph isomorphic to $K_{r, r}$, where $K_{r, r}$ is a complete bipartite graph with $r$ vertices in both sides of the bipartition. 

\begin{theorem}\label{t:gw}
{\em (\cite{Gurski2000TheKnn})}
If $G$ is a graph that has no subgraph isomorphic to $K_{r,r}$, then $tw(G)+1\leq 3(r-1)(2^{rw(G)+1}-1)$.
\end{theorem}

Finally, we show that the treewidth of $G$ is bounded by a function of $\delta$.

\begin{corollary}
\label{cor:no_sc_bounded_tw}
If $G$ is a $C_4$-free odd-signable graph with maximum degree $\delta$ and no star cutset then $tw(G)\leq 45\delta-1$.
\end{corollary}
\begin{proof}
Follows from Corollary \ref{c:rw} and Theorem \ref{t:gw}.
\end{proof}

\section{Balanced separators in $C_4$-free odd-signable graphs}
\label{sec:final}

Let $\delta$ be a positive integer and let $G$ be a $C_4$-free odd-signable graph with maximum degree $\delta$. In this section, we prove Theorem \ref{thm:mainthm}, showing that $G$ has a balanced separator. We begin by stating a helpful lemma showing that if $G$ has bounded treewidth, then $G$ has a balanced separator. 

\begin{lemma}[\cite{Cygan2015}, Lemma 7.19]
\label{lemma:tw_sep}
Let $G$ be a graph with treewidth at most $k$ and let $w:V(G) \to [0, 1]$ be a weight function of $G$ with $w(G) = 1$. Then, $G$ has a $(w, \frac{1}{2}, k+1)$-balanced separator. 
\end{lemma}
Now, we prove that if $G$ is a clean $C_4$-free odd-signable graph with maximum degree $\delta$, then $G$ has a balanced separator.

\begin{theorem}
Let $\delta, d$ be positive integers, let $c \in [\frac{1}{2}, 1)$, let $m \in [0, 1]$, and let $f(2, \delta) = 2(\delta + 1)^2 + 1$, with $d \geq 47\delta + {  2}f(2, \delta)\delta -2$, and $(1-c) + [m + 2f(2, \delta)\delta 2^{\delta}(1-c) + (\delta - 1)2^{\delta}(1-c)](\delta + \delta^2) < \frac{1}{2}$. Let $G$ be a connected clean $C_4$-free odd-signable graph with maximum degree $\delta$ and  {let $w:V(G) \to [0, 1]$ be a weight function on $G$ with $w(G) = 1$ and $w^{\max} \leq m$}. Then, $G$ has a $(w, c, d)$-balanced separator.
\label{thm:clean_balanced_sep}
\end{theorem}
\begin{proof}
Suppose that $G$ does not have a $(w, c, d)$-balanced separator. Let $\mathcal{T}$ be the set of all twin wheel forcers in $G$ and let $\beta_{2k+1}$ be the terminal bag of a $\mathcal{T}$-decomposition of $G$, with $k \leq f(2, \delta)$. It follows from Theorem \ref{thm:bigthm} that $\beta_{2k+1}$ does not have a clique cutset or a $(w', c, d - 2k\delta - 2(\delta - 1))$-balanced separator for some weight function $w'$ with 
{  $w'(\beta_{2k+1})=1$ and}
$w'^{\max} \leq w^{\max} + f(2, \delta)\delta 2^{\delta}(1-c) + (\delta - 1)2^{\delta}(1-c)$. By Theorem \ref{thm:no_twin_wheel}, $\beta_{2k+1}$ does not contain a terminal twin wheel.

By Lemma \ref{lemma:no_sc}, $\beta_{2k+1}$ has no star cutset. Since $\beta_{2k+1}$ has no star cutset, it follows from Corollary \ref{cor:no_sc_bounded_tw} that $\tw(\beta_{2k+1}) \leq 45\delta - 1$. By Lemma \ref{lemma:tw_sep}, $\beta_{2k+1}$ has a $(w', \frac{1}{2}, 45\delta)$-balanced separator. Since $d - 2k\delta - 2(\delta - 1) \geq d - 2f(2, \delta)\delta - 2(\delta - 1) \geq 45\delta$ and $c \geq \frac{1}{2}$, it follows that $\beta_{2k+1}$ has a $(w', c, d - 2k\delta - 2(\delta - 1))$-balanced separator, a contradiction. 
\end{proof}

Finally, we prove Theorem \ref{thm:mainthm}. 
\mainthm*
\begin{proof}
Suppose that $G$ does not have a $(w, c, d)$-balanced separator. Let $\mathcal{F}$ be the set of all strong forcers of $G$ and let $\beta_{2k+1}$ be the terminal bag for an $\mathcal{F}$-decomposition of $G$, with $k \leq f(2, \delta)$. By Theorem \ref{thm:bigthm}, $\beta_{2k+1}$ {  is connected and} does not have a $(w', c, d - 2k\delta - 2(\delta - 1))$-balanced separator for some weight function $w'$ with {  $w'(\beta_{2k+1})=1$ and} $w'^{\max} \leq w^{\max} + f(2, \delta)\delta 2^{\delta}(1-c) + (\delta - 1)2^{\delta}(1-c)$, and by Theorem \ref{thm:strong_forcers_clean}, $\beta_{2k+1}$ is {  connected and} clean. Since $\beta_{2k+1}$ is clean, it follows from Theorem \ref{thm:clean_balanced_sep} that $\beta_{2k+1}$ has a $(w', c, d-2k\delta -2 (\delta - 1))$-balanced separator, a contradiction. \end{proof}

\section*{Acknowledgment}
We thank Cemil Dibek for reading and providing feedback on this paper.

\bibliographystyle{acm}

\bibliography{references}

\begin{thebibliography}{10}

\bibitem{Aboulker2020OnGraphs}
{\sc Aboulker, P., Adler, I., Kim, E.~J., Sintiari, N. L.~D., and Trotignon,
  N.}
\newblock {On the tree-width of even-hole-free graphs}.
\newblock {\em European Journal of Combinatorics 98}, 103394 (2021).

\bibitem{wallpaper}
{\sc Abrishami, T., Chudnovsky, M., Dibek, C., Hajebi, S., Rz\k{a}\.{z}ewski,
  P., Spirkl, S., and Vu\v{s}kovi\'c, K.}
\newblock Induced subgraphs and tree decompositions ii. toward walls and their
  line graphs in graphs of bounded degree.
\newblock {\em arXiv:2108.01162\/} (2021).

\bibitem{Addario-Berry2008BisimplicialGraphs}
{\sc Addario-Berry, L., Chudnovsky, M., Havet, F., Reed, B., and Seymour, P.}
\newblock {Bisimplicial vertices in even-hole-free graphs}.
\newblock {\em Journal of Combinatorial Theory, Series B 98}, 6 (2008),
  1119--1164.

\bibitem{Adler2018}
{\sc Adler, I., and Harwath, F.}
\newblock {Property testing for bounded degree databases}.
\newblock In {\em STACS 2018: 35th Symposium on Theoretical Aspects of computer
  Science\/} (2018), pp.~6:1--6:14.

\bibitem{Bodlaender1988DynamicTreewidth}
{\sc Bodlaender, H.~L.}
\newblock {Dynamic programming on graphs with bounded treewidth}.
\newblock Springer, Berlin, Heidelberg, 1988, pp.~105--118.

\bibitem{Chang2015EHFrecognition}
{\sc Chang, H.-C., and Lu, H.-I.}
\newblock {A faster algorithm to recognize even-hole-free graphs}.
\newblock {\em Journal of Combinatorial Theory, Series B 113\/} (2015),
  141--161.

\bibitem{Conforti1999EvenGraphs}
{\sc Conforti, M., Cornu{\'{e}}jols, G., Kapoor, A., and Vu{\v{s}}kovi{\'{c}},
  K.}
\newblock {Even and odd holes in cap‐free graphs}.
\newblock {\em Journal of Graph Theory 30}, 4 (1999), 289--308.

\bibitem{Conforti2002Even-hole-freeTheorem}
{\sc Conforti, M., Cornu{\'{e}}jols, G., Kapoor, A., and Vu{\v{s}}kovi{\'{c}},
  K.}
\newblock {Even-hole-free graphs part I: Decomposition theorem}.
\newblock {\em Journal of Graph Theory 39}, 1 (2002), 6--49.

\bibitem{Conforti2002Even-hole-freeRecognition}
{\sc Conforti, M., Cornu{\'{e}}jols, G., Kapoor, A., and Vu{\v{s}}kovi{\'{c}},
  K.}
\newblock {Even-hole-free graphs part II: Recognition algorithm}.
\newblock {\em Journal of Graph Theory 40\/} (2002), 238--266.

\bibitem{Cygan2015}
{\sc Cygan, M., Fomin, F.~V., Kowalik, L., Lokshtanov, D., Marx, D., Pilipczuk,
  M., Pilipczuk, M., and Saurabh, S.}
\newblock {\em {Parameterized algorithms.}}
\newblock Springer, 2015.

\bibitem{daSilva2007TriangulatedGraphs}
{\sc da~Silva, M.~V., and Vu{\v{s}}kovi{\'{c}}, K.}
\newblock {Triangulated neighborhoods in even-hole-free graphs}.
\newblock {\em Discrete Mathematics 307}, 9-10 (2007), 1065--1073.

\bibitem{daSilva2013Decomposition2-joins}
{\sc da~Silva, M.~V., and Vu{\v{s}}kovi{\'{c}}, K.}
\newblock {Decomposition of even-hole-free graphs with star cutsets and
  2-joins}.
\newblock {\em Journal of Combinatorial Theory. Series B 103}, 1 (2013),
  144--183.

\bibitem{Gurski2000TheKnn}
{\sc Gurski, F., and Wanke, E.}
\newblock {The tree-width of clique-width bounded graphs without $K_{n,n}$}.
\newblock In {\em Lecture Notes in Computer Science (including subseries
  Lecture Notes in Artificial Intelligence and Lecture Notes in
  Bioinformatics)\/} (2000), vol.~1928, Springer Verlag, pp.~196--205.

\bibitem{Harvey2017ParametersTreewidth}
{\sc Harvey, D.~J., and Wood, D.~R.}
\newblock {Parameters Tied to Treewidth}.
\newblock {\em Journal of Graph Theory 84}, 4 (2017), 364--385.

\bibitem{Thorup2020}
{\sc Lai, K.~Y., Lu, H.~I., and Thorup, M.}
\newblock {Three-in-a-tree in near linear time}.
\newblock In {\em STOC 2020: Proceedings of the 52nd Annual ACM SIGACT
  Symposium on Theory of Computing\/} (2020), pp.~1279--1292.

\bibitem{Le2017DetectingClasses}
{\sc Le, N.~K.}
\newblock {\em {Detecting and Coloring some Graph Classes}}.
\newblock PhD thesis, 2017.

\bibitem{Le2018ColoringCutset}
{\sc Le, N.~K.}
\newblock {Coloring even-hole-free graphs with no star cutset}.
\newblock {\em arXiv:1805.01948\/} (2018).

\bibitem{rankwidth}
{\sc Oum, S.-i., and Seymour, P.}
\newblock {Approximating clique-width and branch-width}.
\newblock {\em Journal of Combinatorial Theory, Series B 96\/} (2006),
  514--528.

\bibitem{Robertson1991GraphTree-decomposition}
{\sc Robertson, N., and Seymour, P.~D.}
\newblock {Graph minors. X. Obstructions to tree-decomposition}.
\newblock {\em Journal of Combinatorial Theory, Series B 52}, 2 (1991),
  153--190.

\bibitem{Silva2010}
{\sc Silva, A., da~Silva, A.~A., and Sales, C.~L.}
\newblock {A bound on the treewidth of planar even-hole-free graphs}.
\newblock {\em Discrete Applied Mathematics 158\/} (2010), 1229 -- 1239.

\bibitem{SD}
{\sc Sintiari, N. L.~D., and Trotignon, N.}
\newblock {(Theta, triangle)-free and (even hole, $K_4$)-free graphs. Part 1 :
  Layered wheels}.
\newblock {\em Journal of Graph Theory 97\/} (2021), 475--509.

\bibitem{Trotignon2012Combinatorial2-joins}
{\sc Trotignon, N., and Vu{\v{s}}kovi{\'{c}}, K.}
\newblock {Combinatorial optimization with 2-joins}.
\newblock {\em Journal of Combinatorial Theory. Series B 102}, 1 (2012),
  153--185.

\bibitem{Vuskovic2010}
{\sc Vu{\v{s}}kovi{\'{c}}, K.}
\newblock {Even-hole-free graphs: A survey}.
\newblock {\em Applicable Analysis and Discrete Mathematics 4\/} (2010),
  219--240.

\end{thebibliography}

\end{document}